\documentclass[a4paper,10pt]{amsart}
\usepackage[utf8]{inputenc}
\usepackage[T1]{fontenc}

\usepackage{amsmath}
\usepackage{amssymb}
\usepackage{amsthm}
\usepackage{mathtools}

\usepackage{eucal} 
\usepackage{mathdots}
\usepackage[colorlinks=true,linkcolor=red,citecolor=red,urlcolor=blue]{hyperref}
\usepackage[capitalise]{cleveref}
\usepackage{enumitem}
\usepackage{mathrsfs}  

\usepackage{tikz-cd}
\usepackage{adjustbox}
\usepackage{contour}
\usepackage{enumitem}

\crefname{equation}{}{}

\newtheorem{theorem}{Theorem}[section]
\newtheorem{innercustomthm}{Theorem}
\newenvironment{customthm}[1]
  {\renewcommand\theinnercustomthm{#1}\innercustomthm}
  {\endinnercustomthm}
\newtheorem{prop}[theorem]{Proposition}
\newtheorem{cor}[theorem]{Corollary}
\newtheorem{lemma}[theorem]{Lemma}

\theoremstyle{definition}
\newtheorem{dfn}[theorem]{Definition}
\newtheorem{rmk}[theorem]{Remark}
\newtheorem{ex}[theorem]{Example}

\DeclareMathOperator{\Hom}{\mathsf{Hom}}

\DeclareMathOperator{\Ext}{\mathsf{Ext}}
\DeclareMathOperator{\Tor}{\mathsf{Tor}}

\DeclareMathOperator{\RHom}{\mathbf{R}\mathsf{Hom}}

\newcommand{\Acal}{\mathcal{A}}
\newcommand{\Bcal}{\mathcal{B}}
\newcommand{\Ccal}{\mathcal{C}}
\newcommand{\Dcal}{\mathcal{D}}

\newcommand{\Gcal}{\mathcal{G}}

\newcommand{\Kcal}{\mathcal{K}}
\newcommand{\Lcal}{\mathcal{L}}

\newcommand{\Ocal}{\mathcal{O}}

\newcommand{\Scal}{\mathcal{S}}
\newcommand{\Tcal}{\mathcal{T}}

\newcommand{\Xcal}{\mathcal{X}}

\newcommand{\Qbb}{\mathbb{Q}}

\newcommand{\Zbb}{\mathbb{Z}}

\newcommand{\D}{\mathbf{D}}

\newcommand{\cpt}{\mathsf{c}}
\newcommand{\clp}{\mathsf{clp}}

\newcommand{\op}{\mathsf{op}}

\newcommand{\fp}[1]{\mathsf{fp}(#1)} 
\newcommand{\Mod}[1]{\mathsf{Mod}\mbox{-}#1}

\newcommand{\Spec}[1]{\mathsf{Spec}(#1)}
\newcommand{\SH}{\mathcal{SH}}
\newcommand{\Spech}[1]{\mathsf{Spec}^\mathsf{h}(#1)}

\newcommand{\Supp}{\mathsf{Supp}}
\newcommand{\SuppBFS}{\mathsf{Supp}_{\textnormal{BFS}}}
\newcommand{\SuppN}{\mathsf{Supp}_{\textnormal{Naï}}}
\newcommand{\supp}{\mathsf{supp}}
\newcommand{\supph}{\mathsf{supp}^\mathsf{h}}

\newcommand*{\Perp}[1]{{}^{\perp_{#1}}}
\newcommand{\hocolim}{\mathsf{hocolim}}
\newcommand{\colim}{\mathsf{colim}}
\newcommand{\holim}{\mathsf{holim}}

\newcommand{\Loc}[1]{\langle #1 \rangle^{\oplus}}
\newcommand{\DefBW}[1]{\mathsf{Def}^{\otimes}(#1)}
\newcommand{\Def}[1]{\langle #1 \rangle^{\textsf{def}}}
\newcommand{\Deft}[1]{\langle #1 \rangle^{\textsf{def}}}
\newcommand{\thick}[1]{\langle #1 \rangle}

\newcommand{\Ker}{\mathsf{Ker}}

\newcommand{\pp}{\mathfrak{p}}
\newcommand{\PP}{\mathfrak{p}}
\newcommand{\MM}{\mathfrak{m}}

\newcommand{\qq}{\mathfrak{q}}
\newcommand{\QQ}{\mathfrak{q}}

\newcommand{\mm}{\mathfrak{m}}
\newcommand{\nn}{\mathfrak{n}}

\newcommand{\qc}{\mathsf{qc}}

\newcommand{\yo}{\mathbf{y}}
\newcommand{\toeq}{\xrightarrow{\cong}}

\newcommand{{\tst}}{\textit{t}-}
\newcommand{\TC}{\mathsf{(TC)}}
\newcommand{\RTC}{\mathsf{(RTC)}}
\newcommand{\NoSC}{\mathsf{(NoSC)}}
\newcommand{\SLP}{\mathsf{(SLP)}}
\newcommand{\HSLP}{\mathsf{(HSLP)}}

\newcommand{\newterm}[1]{\textit{#1}}

\title[Telescope conjecture via homological residue fields]{Telescope conjecture via homological residue fields with applications to schemes}
\author{Michal Hrbek}
\address[M. Hrbek]{Institute of Mathematics of the Czech Academy of Sciences, \v{Z}itn\'{a} 25, 115 67 Prague, Czech Republic}
\email{hrbek@math.cas.cz}


\thanks{The author was supported by the GAČR project 23-05148S and the Academy of Sciences of the Czech Republic (RVO 67985840).}
\begin{document}
\begin{abstract}
    For a big tt-category, we give a characterization of the Telescope Conjecture $\TC$ in terms of definable $\otimes$-ideals generated by homological residue fields. We formulate a stalk-locality property of $\TC$ and prove that it holds in the case of the derived category of a quasi-compact quasi-separated scheme, strengthening a result \cite{HHZ21}. As an application, we find strong links between $\TC$ and separation properties of the adic topology on local rings. This allows us to recover known examples and counterexamples of when $\TC$ holds over a scheme, as well as to construct some new ones.
\end{abstract}
\maketitle
\tableofcontents
\section*{Introduction}
The Telescope Conjecture $\TC$ was originally formulated by Ravenel \cite{Rav84} and asks whether the smashing localizations of the stable homotopy category of spectra $\SH$ given by the Morava E-theories coincide with those given by telescopes of finite spectra. The $\TC$ for $\SH$ has had a long and rocky history, but has been recently settled in the negative by Burklund, Hahn, Levy, and Schlank \cite{BHLS23}. In an equivalent formulation, $\TC$ asks if every smashing localization of $\SH$ is equivalent to a localization away from a set of finite spectra, a property formulated by Bousfield \cite{Bou79}. In this way, it makes sense to study $\TC$ also in other stable homotopy categories, e.g. the derived categories of schemes. There, it turns out that $\TC$ is a property rather than a conjecture, as it is satisfied by some schemes including all of the noetherian ones (as shown by Neeman \cite{Nee92} in the affine case, and Alonso, Jeremías, and Souto \cite{AJS04} in the non-affine case), but fails in general (Keller's counterexample \cite{Kel94}).

Motivated by the results in algebraic geometry, a supply of locality results for $\TC$ has been discovered, allowing to establish $\TC$ in more general settings. Balmer and Favi formulated $\TC$ in the generality of tensor triangular geometry and proved that it is an affine-local property, that is, $\TC$ can be checked locally on covers of the Balmer spectrum by quasi-compact open sets. In an algebro-geometric setting, Antieau \cite{Ant14} later even promoted this to an étale-locality. In particular, the study of $\TC$ in the derived category $\D(X)$ of a quasi-compact, quasi-separated scheme $X$ reduces to the case of an affine scheme. Hu, Zhu, and the author found in \cite{HHZ21} that $\TC$ in $\D(X)$ is even a stalk-local property, that is, its validity depends only on it being satisfied in the derived categories $\D(\Ocal_x)$ of the stalks, reducing the problem to the case of a local commutative ring. In one of the main results of this paper, we prove the following refinement.
\begin{customthm}{A}[\cref{tc-scheme}]\label{thmA}
  Let $X$ be a quasi-compact and quasi-separated scheme. Then the following are equivalent:
  \begin{enumerate}
    \item[(i)] $\TC$ holds for $\D(X)$,
    \item[(ii)] for each $x \in X$, the residue field $k(x)$ generates $\D(\Ocal_x)$ as a definable $\otimes$-ideal.
  \end{enumerate}
\end{customthm}
The computation of the definable $\otimes$-ideal generated by the residue field of $k$ of a local ring $(R,\mm,k)$ is in tight connection with the properties of the $\mm$-adic topology on $R$. In fact, essentially by Keller's example, if $\Def{k}=\D(R)$ holds in $\D(R)$ then $R$ is $\mm$-adically transfinitely separated (\cref{sep-tc}). On the other hand, $\Def{k}=\D(R)$ is implied by a stronger separation property of $R$ (\cref{psep-tc}), enjoyed e.g. by all local noetherian rings. Although we do not obtain an easy ring-theoretic criterion for a local ring to satisfy $\TC$, our method recovers most, if not all, of the known examples and counterexamples for validity of $\TC$ in $\D(R)$, including the results of Neeman \cite{Nee92}, Stevenson \cite{Ste14}, Bazzoni and Šťovíček \cite{BS17}, and Dwyer and Palmieri \cite{DP08}. We also add to the list some new examples, including a local separated ring $R$ such that $\TC$ fails in $\D(R)$ (\cref{ex-1}). In Section 5, we formulate a restricted version of the Telescope Conjecture, which holds for a commutative ring $R$ if and only if all pseudoflat ring epimorphisms over $R$ are flat (\cref{flat-tc}). In this way, \cref{ex-1} also yields an interesting example of a pseudoflat non-surjective local epimorphism $R \to S$ of commutative rings.

Before we pass to study the algebro-geometric setting, we also establish some results in the general setting of a big tt-category $\Tcal$. First, it is essential for our approach that we take an alternative, but equivalent, viewpoint on $\TC$ which focuses on definable $\otimes$-ideals instead of the smashing $\otimes$-ideals. This formulation of $\TC$ based on model theory of compactly generated triangulated categories was discovered by Krause \cite{Kr00}. In fact, the theory of purity and definable subcategories has recently found more development and applications to tt-geometry, see Bird and Williamson \cite{BW23, BW23b}, Wagstaffe \cite{WPHD}, or Prest and Wagstaffe \cite{PW23}. The purity theory also plays an important role in the recent works of Balmer, Krause, and Stevenson \cite{BKS19} and Balmer \cite{Bal20}, where a theory of homological residue fields is developed for big tt-categories, which specializes to the usual residue field objects $k(x)$ in case of $\D(X)$ and to Morava K-theories in the case of $\SH$. Assuming that Balmer's ``Nerves of Steel'' Conjecture holds for $\Tcal$, we show in \cref{tc-general} by an application of Balmer's Tensor Nilpotence Theorem of \cite{Bal20} that it is sufficient to check $\TC$ in $\Tcal$ just on the definable $\otimes$-ideals generated by the homological residue fields lying over each Thomason subset of the Balmer spectrum. In Section 3, we formulate the stalk-locality principle for $\TC$ in $\Tcal$ and characterize in \cref{cg-stalklocal} the situation in which the stalk-locality holds in each compact localization of $\Tcal$. We obtain the following formulation of \cref{thmA} in a general big tt-category.
\begin{customthm}{B}[\cref{tc-stalklocal}]\label{thmB}
  Let $\Tcal$ be a big tt-category satisfying the ``Nerves of Steel'' Conjecture and certain stalk-locality principle formulated in \cref{ss-stalkloc}. Then the following are equivalent:
  \begin{enumerate}
    \item[(i)] $\TC$ holds for $\Tcal$,
    \item[(ii)] for each $\pp \in \Spec{\Tcal^\cpt}$, the homological residue field object $E_\pp$ generates the stalk tt-category $\Tcal_\pp$ as a definable $\otimes$-ideal.
  \end{enumerate}
\end{customthm}

We do not know if the stalk-locality principle holds in all big tt-categories. Apart from $\D(X)$ (\cref{biloc-stalks}), it also holds in $\SH$ (see \cref{exSH} for the interpretation of \cref{thmB} in this setting), and more generally, whenever the Balmer-Favi-Sanders support detects vanishing in $\Tcal$ (\cref{exBFS}). It is another open question whether the Balmer-Favi-Sanders support always detects vanishing. In fact, this is open even in the case $\D(X)$ for a general $X$ and our proof of \cref{biloc-stalks} indeed takes a different route.

\subsection*{Acknowledgement} The author would also like to express thanks to Isaac Bird, Sergio Pavon, Enrico Sabatini, and Jordan Williamson for useful discussions regarding the paper. The author is grateful to the anonymous referee for a careful reading of the manuscript and providing a lot of very useful feedback.
\section{Definable $\otimes$-ideals and Telescope Conjecture}
In the next three sections, we will work in the generality of a tensor triangulated category in the sense of Balmer \cite{Bal05, Bal1, BF11}, some basic facts about which we briefly recall now. 
\subsection{Small tt-categories} Let $\Kcal$ be a triangulated category with suspension functor $\Sigma$. We say that $\Kcal$ is a \newterm{tensor triangulated category (tt-category)} if it is equipped with a symmetric monoidal product $- \otimes -: \Kcal \times \Kcal \to \Kcal$ with unit $1 \in \Kcal$ and such that for each $x \in \Kcal$, the functor $x \otimes -: \Kcal \to \Kcal$ is a triangulated functor. All subcategories in this paper are automatically full, additive, and isomorphism-closed. A \newterm{thick $\otimes$-ideal} of a tt-category $\Kcal$ is a thick subcategory $\Scal$ such that $x \otimes y \in \Scal$ for any $x \in \Kcal$ and $y \in \Scal$. We call a thick $\otimes$-ideal $\Scal$ \newterm{prime} if for any $x,y \in \Kcal$ such that $x \otimes y \in \Scal$, either $x \in \Scal$ or $y \in \Scal$.  To a skeletally small tt-category $\Kcal$, Balmer associated a topological space $\Spec{\Kcal}$ we call the \newterm{Balmer spectrum}. By definition, this is the set of all prime thick $\otimes$-ideals in $\Kcal$. The \newterm{support} of an object $x \in \Kcal$ is defined as $\supp(x) = \{\PP \in \Spec{\Kcal} \mid x \not\in \PP\}$. The crucial property of the theory is that this support data classifies the radical thick $\otimes$-ideals of $\Kcal$. A thick $\otimes$-ideal $\Scal$ is \newterm{radical} if $\otimes^n x \in \Scal$ for some $n>0$ implies $x \in \Scal$ for any $x \in \Kcal$; every thick $\otimes$-ideal is radical if every object of $\Kcal$ is rigid (see below), which will always be the case for compact objects in our further setting. The Balmer spectrum is a topological space whose topology is given by a basis of closed subsets of the form $\supp(x)$ for all $x \in \Kcal$, this makes $\Spec{\Kcal}$ into a spectral space. A subset $V$ of $\Spec{\Kcal}$ is called \newterm{Thomason} if it can be written as a union of closed subsets with quasi-compact complements in $\Spec{\Kcal}$. Then the assignment $V \mapsto \Kcal_V = \{x \in \Kcal \mid \supp(x) \subseteq V\}$ yields a bijection between Thomason subsets of $\Spec{\Kcal}$ and the radical thick $\otimes$-ideals in $\Kcal$. The main reference is \cite{Bal05} here.

\subsection{Big tt-categories} Let $\Tcal$ be a \newterm{big tt-category}, by which we mean a rigidly-compactly generated tensor triangulated category, the base setting of \cite{BF11}. By definition, this means that $\Tcal$ has the following attributes: First, it is a compactly generated triangulated category, that is, $\Tcal$ is cocomplete and the subcategory $\Tcal^\cpt$ of its compact objects is skeletally small and generates $\Tcal$. Next, $\Tcal$ is a tt-category and $\Tcal^\cpt$ is its tt-subcategory, meaning that $- \otimes -$ restricts to $\Tcal^\cpt$ and $1 \in \Tcal^\cpt$. Finally, we assume $- \otimes -$ to be closed, which says that $X \otimes -$ admits a right adjoint $[X,-]: \Tcal \to \Tcal$ for all $X \in \Tcal$, and that each $x \in \Tcal^\cpt$ is rigid in the sense that the natural map $[x,1] \otimes Y \to [x,Y]$ is an isomorphism for any $Y \in \Tcal$. In particular, the contravariant endofunctor $(-)^* := [-,1]$ of $\Tcal$ restricts to a duality $(\Tcal^\cpt)^\op \toeq \Tcal^\cpt$. Next we muster the following bestiary of certain subcategories of a big tt-category $\Tcal$. A thick $\otimes$-ideal $\Ccal$ of $\Tcal$ is called:
\begin{itemize}
  \item a \newterm{localizing $\otimes$-ideal}, if $\Ccal$ is closed under all coproducts;
  \item a \newterm{strict localizing $\otimes$-ideal} if $\Ccal$ is localizing and the inclusion $\Ccal \xhookrightarrow{} \Tcal$ has a right adjoint;
  \item a \newterm{smashing $\otimes$-ideal} if $\Ccal$ is strict localizing and the right adjoint to the inclusion  $\Ccal \xhookrightarrow{} \Tcal$ preserves coproducts;
  \item a \newterm{definable $\otimes$-ideal}, if $\Ccal$ is a definable subcategory, that is, there is a set $\Phi$ of morphisms in $\Tcal^\cpt$ such that $\Ccal = \{X \in \Tcal \mid \Hom_\Tcal(f,X) = 0 ~\forall f \in \Phi\}$.
\end{itemize}
A subcategory $\Ccal$ of $\Tcal$ is called a \newterm{thick $\otimes$-coideal} if it is a thick subcategory closed under $[X,-]$ for any $X \in \Tcal$. A $\otimes$-coideal $\Ccal$ is called:
\begin{itemize}
 \item a \newterm{colocalizing $\otimes$-coideal}, if $\Ccal$ is closed under all products;
  \item a \newterm{strict colocalizing $\otimes$-coideal} if $\Ccal$ is colocalizing and the inclusion $\Ccal \xhookrightarrow{} \Tcal$ has a left adjoint;
  \item a \newterm{cosmashing $\otimes$-coideal} if $\Ccal$ is strict localizing and the left adjoint to the inclusion  $\Ccal \xhookrightarrow{} \Tcal$ preserves products.
\end{itemize}
\subsection{} Given a subcategory $\Ccal$ of $\Tcal$, we let $\thick{\Ccal}$ denote the smallest thick $\otimes$-ideal of $\Tcal$ containing $\Ccal$, $\Loc{\Ccal}$ denote the smallest localizing $\otimes$-ideal containing $\Ccal$, and $\Deft{\Ccal}$ denote the smallest definable $\otimes$-ideal containing $\Ccal$, the latter closure operator is well-defined using \cite[Theorem 5.2.14]{WPHD}. If $\Ccal = \{X\}$ for some object $X \in \Tcal$, we drop the curly brackets. A \newterm{semiorthogonal $\otimes$-decomposition} of $\Tcal$ is a pair $(\Lcal,\Ccal)$ of thick $\otimes$-ideals such that $\Hom_{\Tcal}(\Lcal,\Ccal) = 0$ and $\Tcal = \Lcal \star \Ccal$, where 
$$\Lcal \star \Ccal = \{X \in \Tcal \mid \exists \text{ triangle } L \to X \to C \mathrel{\leadsto} \text{ with } L \in \Lcal, C \in \Ccal\}.$$
If $\Scal$ is a subcategory of $\Tcal$ then denote 
$$\Scal\Perp{} = \{X \in \Tcal \mid \Hom_\Tcal(S,\Sigma^i X) = 0 ~\forall S \in \Scal, ~\forall i \in \Zbb\},$$ 
$$\Perp{}\Scal = \{X \in \Tcal \mid \Hom_\Tcal(X,\Sigma^i S) = 0 ~\forall S \in \Scal, ~\forall i \in \Zbb\}.$$  
If $(\Lcal,\Ccal)$ is a semiorthogonal $\otimes$-decomposition then $\Ccal = \Lcal\Perp{}$ and $\Lcal = \Perp{}\Ccal$. The following proposition is standard, see \cite[Remark 2.11]{BCHS23} and \cite[Theorem 2.6]{BF11}.
\begin{prop}\label{loc-coloc}
  The following collections are in mutual bijection:
  \begin{enumerate}
    \item[(i)] strict localizing $\otimes$-ideals $\Lcal$ of $\Tcal$,
    \item[(ii)] strict colocalizing $\otimes$-coideals $\Ccal$ of $\Tcal$, 
    \item[(iii)] semiorthogonal $\otimes$-decompositions $(\Lcal,\Ccal)$ of $\Tcal$. 
  \end{enumerate}
\end{prop}
\subsection{Definable $\otimes$-ideals}
As definable $\otimes$-ideals will play a crucial role, let us gather some facts about them here. Following Krause \cite{Kr00}, $\Tcal$ admits a theory of purity encoded in terms of the \newterm{restricted Yoneda functor} $\yo$. Let $\Acal = \Mod{\Tcal^\cpt}$ be the category of right $\Tcal^\cpt$-modules, that is, the category of additive functors $(\Tcal^\cpt)^\op \to \Mod \Zbb$. Then $\yo: \Tcal \to \Acal$ is defined by the rule $X \to \Hom_{\Tcal}(-,X)_{\restriction \Tcal^\cpt}$. The tensor product on $\Tcal$ extends to a unique coproduct-preserving symmetric monoidal product $- \otimes -$ on the Grothendieck category $\Acal$ such that $\yo x \otimes \yo y = \yo(x \otimes y)$ for all $x,y \in \Tcal$. This monoidal product is in addition closed, \cite{Bal20b}, let $[-,-]_\Acal: \Acal \to \Mod \Zbb$ denote the internal hom functor. Note that unlike for the tensor product, we distinguish the internal homs $[-,-]$ of $\Tcal$ and $[-,-]_\Acal$ of $\Acal$ in notation because $\yo$ is usually not a closed functor. We also know that $\Acal$ is a locally coherent category so that the subcategory $\fp \Acal$ of its finitely presentable objects is an abelian subcategory. A morphism $f$ in $\Tcal$ is a \newterm{pure monomorphism} (resp., \newterm{pure epimorphism}) if $\yo f$ is a monomorphism (resp., epimorphism) in $\Acal$. We say that a subcategory $\Ccal$ of $\Tcal$ is closed under pure monomorphisms if for any pure monomorphism $X \to Y$ with $Y \in \Ccal$ we have $X \in \Ccal$; similarly we talk about closure under pure epimorphisms. An object $X \in \Tcal$ is \newterm{pure-injective} if $\yo X$ is injective in $\Acal$, any injective object in $\Acal$ is of this form. A subcategory $\Ccal$ of $\Tcal$ is \newterm{definable} if it is of the form $\Ccal = \{X \in \Tcal \mid \Hom_\Tcal(f,X) = 0 ~\forall f \in \Phi\}$ for a set of morphisms $\Phi$ between compact objects of $\Tcal$. 

We say that $\Tcal$ \newterm{has a model} if it is the homotopy category of a monoidal model category.
\begin{prop}\label{def-ideal}
  For a subcategory $\Dcal$ of $\Tcal$, the following are equivalent:
  \begin{enumerate}
    \item[(i)] $\Dcal$ is a definable $\otimes$-ideal,
    \item[(ii)] there is a smashing $\otimes$-ideal $\Lcal$ such that $\Dcal = \Lcal\Perp{}$,
    \item[(iii)] $\Dcal$ is simultaneously a localizing $\otimes$-ideal and a strict colocalizing $\otimes$-coideal.
  \end{enumerate}
  In addition, consider the condition
  \begin{enumerate}
    \item[(iv)]$\Dcal$ is a thick $\otimes$-ideal closed under products and pure monomorphisms.
  \end{enumerate}
  Then $(i) \implies (iv)$ and the converse is true if $\Tcal$ has a model.
\end{prop}
\begin{proof}
  The bijection between $(i)$ and $(ii)$ is proved in Wagstaffe's thesis \cite[Proposition 5.2.13]{WPHD}, as a restriction of Krause's theory \cite{Kr00} to those smashing subcategories which are $\otimes$-ideals. The equivalence $(ii) \iff (iii)$ follows by \cite[Proposition 5.5.1]{Kra10} together with the proof of \cite[Proposition 5.2.13]{WPHD}. Namely, if $\Lcal$ is a smashing $\otimes$-ideal then $\Lcal\Perp{}$ is a strict $\otimes$-colocalizing coideal by the implication $(i)\implies(ii)$ of \cref{loc-coloc} and a localizing $\otimes$-ideal by the definition and \cite[Theorem 5.1.8]{WPHD}. Conversely, if $(iii)$ is satisfied then $\Lcal=\Perp{}\Dcal$ is a strict localizing $\otimes$-ideal which si smashing because $\Dcal$ is coproduct closed and using \cite[Proposition 5.5.1]{Kra10}.
  
  \noindent
  The implication $(i) \implies (iv)$ follows directly from the definition, as the class of objects vanishing under $\Hom_\Tcal(f,-)$ for a map $f$ between compacts has the desired closure properties. If $\Tcal$ has a model, then the implication $(iv) \implies (i)$ follows by the result of Laking \cite[Theorem 3.11]{Lak20}. Indeed, $(iv)$ implies that $\Dcal$ is also closed under coproducts, as for any coproduct $\coprod_{i \in I}X_i$ in $\Tcal$, the natural map $\coprod_{i \in I}X_i \to \prod_{i \in I}X_i$ is a pure monomorphism in $\Tcal$. Since $\Dcal$ is a thick subcategory closed under pure monomorphism it is also closed under pure epimorphisms. Together we obtain that $\Dcal$ is closed under directed homotopy colimits computed in the model, and then the result of Laking applies.
\end{proof}

\subsection{TTF-triples and tensor idempotents} A \newterm{TTF $\otimes$-triple} in $\Tcal$ is a triple $(\Lcal,\Dcal,\Ccal)$ such that both $(\Lcal,\Dcal)$ and $(\Dcal,\Ccal)$ are semiorthogonal $\otimes$-decompositions. Following Balmer-Favi \cite{BF11}, a \newterm{right $\otimes$-idempotent} is a morphism $1 \to \lambda$ such that $\lambda \otimes (1 \to \lambda)$ is an isomorphism. A \newterm{morphism of right $\otimes$-idempotents} $1 \to \lambda$ and $1 \to \lambda'$ is a map $\lambda \to \lambda'$ which fits into a commutative triangle. Such a map, if it exists, is unique \cite[Corollary 3.7]{BF11}. We say that right $\otimes$-idempotents $1 \to \lambda$ and $1 \to \lambda'$ are \newterm{equivalent} if there is a morphism between them which is an isomorphism.
\begin{prop}\label{ttftriples}
  The following collections are sets and are in mutual bijection:
  \begin{enumerate}
    \item[(i)] definable $\otimes$-ideals $\Dcal$ of $\Tcal$,
    \item[(ii)] smashing $\otimes$-ideals $\Lcal$ of $\Tcal$,
    \item[(iii)] cosmashing $\otimes$-coideals $\Ccal$ of $\Tcal$,
    \item[(iv)] TTF $\otimes$-triples $(\Lcal,\Dcal,\Ccal)$ of $\Tcal$,
    \item[(v)] right $\otimes$-idempotents $1 \to \lambda$ up to equivalence.
  \end{enumerate}
  The bijection between $(i)$ and $(v)$ assigns to $\Dcal$ the reflection morphism $1 \to \lambda_\Dcal$ with respect to the reflective subcategory $\Dcal$ of $\Tcal$.
\end{prop}
\begin{proof}
  Follows from \cite[Theorem 3.5]{BF11}, \cite[Proposition 6.3]{LV20} together with \cref{def-ideal}.
\end{proof}
\begin{rmk}
  The right $\otimes$-idempotents correspond up to natural equivalence to \newterm{smashing localizations} of $\Tcal$, that is, to Bousfield localization functors $L: \Tcal \to \Tcal$ which are of the form $X \otimes -$ for some object $X \in \Tcal$, see \cite[\S 2,3]{BF11}. The smashing $\otimes$-ideal corresponding to a smashing localization $L$ is just its kernel, while the corresponding definable $\otimes$-ideal is its essential image. One immediate advantage of working with essential images instead of kernels is that definable $\otimes$-ideals are, unlike smashing $\otimes$-ideals, and under the mild assumption of $\Tcal$ having a model, determined by closure properties, see \cref{def-ideal}.
\end{rmk}
\subsection{Compact generation} It is clear that for any thick $\otimes$-ideal $\Scal \subseteq \Tcal^\cpt$, the category $\Dcal = \Scal\Perp{}$ is a definable $\otimes$-ideal in $\Tcal$, the definability is witnessed by the set $\Phi = \{s \xrightarrow{1_s} s \mid s \in \Scal\}$ of identity maps. If $\Dcal$ is of this form, we call it a \newterm{compactly generated} definable $\otimes$-ideal. In view of \cref{ttftriples}, we also call the associated smashing $\otimes$-ideal and the associated TTF $\otimes$-triple compactly generated. Given a Thomason set $V$, we let $\Lcal_V = \Loc{\Kcal_V}$ denote the localizing $\otimes$-ideal in $\Tcal$ generated by $\Kcal_V$. The standard observation dictates that $\Lcal_V \cap \Tcal^\cpt = \Kcal_V$, which immediately yields the following.
\begin{prop}\label{Balm-cg}
  The assignment $V \mapsto (\Lcal_V, \Tcal_V, \Ccal_V)$ induces a bijection between Thomason sets $V$ in $\Spec{\Tcal^\cpt}$ and the set of compactly generated $\otimes$-TTF triples in $\Tcal$.
\end{prop}
In addition, we denote by $1 \to \lambda_V := \lambda_{\Tcal_V}$ the associated right $\otimes$-tensor idempotent. The definable $\otimes$-ideal $\Tcal_V$ is a big tt-category itself, see \cite[Theorem 4.1]{BF11}. The tensor structure on $\Tcal_V$ is given by restriction of $\otimes$ (and of $[-,-]$), the tensor unit is $\lambda_V$, and the compact objects $\Tcal_V^\cpt$ are equivalent to $\thick{\Kcal/\Kcal_V}$. The spectrum $\Spec{\Tcal_V^\cpt}$ is then homeomorphic to the complement subspace $V^c$ of $V$ in $\Spec{\Tcal^\cpt}$, \cite[Proposition 3.11]{Bal05}. 

\subsection{Telescope Conjecture} We say that $\Tcal$ satisfies the \newterm{Telescope Conjecture $\TC$} if every definable $\otimes$-ideal of $\Tcal$ is compactly generated. In other words, $\TC$ is the claim that \cref{Balm-cg} describes \newterm{all} of the TTF $\otimes$-triples in $\Tcal$. Telescope Conjecture was originally formulated in algebraic topology by Ravenel \cite{Rav84} for the stable homotopy category $\SH$ of spectra, there it indeed remained a conjecture until the very answer in the negative by Burklund, Hahn, Levy, and Schlank \cite{BHLS23}. The past and new (counter-)examples coming from algebraic geometry will be the topic of the last three sections of this paper.

\subsection{Right $\otimes$-idempotents as homotopy colimits}
We will employ the following auxiliary result in Section 4 in the special case $\Tcal = \D(R)$. Let $V_1,\ldots,V_n$ be a collection of Thomason subsets of $\Spec{\Tcal^\cpt}$. Then the right $\otimes$-idempotent $1 \to \lambda_V$ corresponding to the union $V = V_1 \cup \ldots \cup V_n$ can by \cite[Proposition 3.11]{BF11} be represented by the tensor product $\bigotimes_{i=1}^n(1 \to \lambda_{V_i}) = 1 \to \bigotimes_{i=1}^n \lambda_{V_i}$. The following provides an infinite union version. 

\begin{theorem}\label{hocolim}
  Let $\Tcal$ be a big tt-category with a model. Let $V=\bigcup_{i\in I}V_i$ be a Thomason subset of $\Spec{\Tcal^\cpt}$ written as a union of Thomason subsets $V_i$. Then the right $\otimes$-idempotent $1 \to \lambda_V$ is a directed homotopy colimit of right $\otimes$-tensor idempotents of the form $1 \to \bigotimes_{i\in F}\lambda_{V_i}$, where $F \subseteq I$ is any finite subset.

  In particular, $1 \to \lambda_V$ is a directed homotopy colimit of right $\otimes$-tensor idempotents of the form $1 \to \lambda_W$, where $W \subseteq V$ is any subset with open quasi-compact complement in $\Spec{\Tcal^\cpt}$.
\end{theorem}
\begin{proof}
  Let $P$ be the poset of all finite subsets of $I$. For each $p \in P$, let $e(p) = \bigotimes_{i \in p} \lambda_{V_i}$, with the convention $e(\emptyset) = 1$. For each inclusion $p \subseteq p'$ we have the unique morphism $(1 \to e(p' \setminus p)) \otimes e(p)$ of right $\otimes$-tensor idempotents $e(p) \to e(p')$, obtaining a directed diagram $\mathscr{D}_P$ of shape $P$ in $\Tcal$. Suppose for now that this diagram in the homotopy category $\Tcal$ lifts to a diagram $\mathscr{E}_P$ in the model. Then we may compute its homotopy colimit $\hocolim \mathscr{E}_P$ in $\Tcal$ by computing the ordinary colimit $\colim \mathscr{C}_P$ in the model, where $\mathscr{C}_P$ is the cofibrant replacement of $\mathscr{E}_P$ with respect to a suitable model structure, see \cite[\S 5.1]{Ho99}. Note that for each $p \subseteq p'$, the corresponding map in $\mathscr{C}_P$ is still isomorphic in $\Tcal$ to $e(p) \to e(p')$, as the weak equivalences in the diagram category are precisely the point-wise weak equivalences. In $\Tcal$, we have the induced map $1 \to \hocolim \mathscr{E}_P$. Since homotopy coherent cones commute with homotopy colimits, this map extends to a triangle $\hocolim \mathscr{F}_P \to 1 \to \hocolim \mathscr{E}_P \to \Sigma \hocolim \mathscr{F}_P$, where $\mathscr{F}_P$ is a diagram whose $p$-th vertex $f(p)$ fits in $\Tcal$ into a triangle $f(p) \to 1 \to e(p) \to \Sigma f(p)$. Since $f(p) \in \Lcal_{\bigcup_{i \in p}V_i} \subseteq \Lcal_V$, we have $\hocolim \mathscr{F}_P \in \Lcal_V$, as localizing $\otimes$-ideals are closed under homotopy colimits. On the other hand, $\colim \mathscr{C}_P$ is isomorphic to $\colim \mathscr{C}_P^i$, where $\mathscr{C}_{P(i)}$ is the same diagram restricted to $P(i) = \{p \in P \mid i \in p\}$. It follows from the definition \cite[\S 5.1]{Ho99} that $\mathscr{C}_{P(i)}$ is cofibrant, and thus $\hocolim \mathscr{F}_P \cong \hocolim \mathscr{F}_{P(i)}$ for any $i \in I$. Therefore, $\hocolim \mathscr{F}_P \in \Dcal = \bigcap_{i \in I}\Dcal_{V_i}$. Clearly, $\Dcal$ is a definable $\otimes$-ideal. In addition, $\Dcal = (\bigcup_{i \in I}\Kcal_{V_i})\Perp{}$, and so $\Dcal$ is compactly generated. By \cref{Balm-cg}, there is a Thomason set $W$ such that $\Dcal = \Dcal_W = \Kcal_W\Perp{}$. Since $V = \bigcup_{i \in I}V_i$, necessarily $W = V$. Finally, by \cite[Theorem 3.5(a)]{BF11} it follows that the map $1 \to \hocolim \mathscr{E}_P$ identifies with the right $\otimes$-idempotent $1 \to \lambda_V$.
  
  \noindent
  It remains to lift $\mathscr{D}_P$ to the model. Fix cofibrant replacements $f_i: 1^c \to \lambda_{V_i}^c$ between cofibrant objects in the model lifting the right $\otimes$-idempotents $1 \to \lambda_{V_i}$ in $\Tcal$. To each inclusion $p \subseteq p'$ we assign the map $(\bigotimes_{i \in (p' \setminus p)}f_i) \otimes \bigotimes_{i \in p}\lambda_{V_i}^c$. Since we are working in a symmetric monoidal category, this defines a direct system $\mathscr{E}_P$ of shape $P$ in the model, which can easily be checked to project onto $\mathscr{D}_P$ in the homotopy category $\Tcal$. The last claim now follows easily as $V$ is the directed union of all of its subsets with open quasi-compact complements.
\end{proof}
\begin{rmk}
  There is a closely related result by Stevenson available in \cite[Lemma 6.6]{St13}, which is in certain sense generalized by \cref{hocolim}. Stevenson's result assumes $\Spec{\Tcal^\cpt}$ to be a noetherian space and addresses a union of a \textit{chain} of Thomason subsets (equivalently, specialization closed subsets in the noetherian situation). In \cref{hocolim}, there is no assumption on $\Spec{\Tcal^\cpt}$ and we do not assume that $V$ is a union of a chain of the subsets, a directed shape that submits more easily to lifting to the model. On the other hand, our construction always produces a coherent directed diagram of shape $P$, the lattice of finite subsets of $I$, ignoring any potential directed set structure on $I$.
\end{rmk}
\section{Telescope conjecture via homological residue fields}\label{S:two}
The goal of this section is to provide a characterization of when $\TC$ is satisfied in $\Tcal$ in terms of definable $\otimes$-ideals generated by the homological residue field objects, as introduced recently by Balmer, Krause, and Stevenson \cite{BKS19} and Balmer \cite{Bal20b}.
\subsection{Homological residue fields} Here we follow \cite{Bal20b}. A subcategory $\Scal$ of $\fp \Acal$ is called a \newterm{Serre $\otimes$-ideal} if it is closed under extensions, subobjects, quotients, and tensoring by any object of $\Acal$. The \newterm{homological spectrum} $\Spech{\Tcal^\cpt}$ of $\Tcal$ is a topological space whose points are maximal proper Serre $\otimes$-ideals of $\fp \Acal$. To each $\Bcal \in \Spech{\Tcal^\cpt}$, we assign the localizing $\otimes$-ideal $\Loc{\Bcal}$ in $\Acal$ generated by $\Bcal$ (equivalently, this is the direct limit closure $\varinjlim \Bcal$). Let $\Acal_\Bcal = \Acal/\Loc{\Bcal}$ be the Serre localization of $\Acal$ at $\Loc{\Bcal}$, and let $\yo_\Bcal: \Tcal \to \Acal_\Bcal$ be the cohomological functor obtained by composing $\yo$ with the localization functor $q_\Bcal: \Acal \to \Acal_\Bcal$. Consider the injective envelope $\yo_\Bcal 1 \to \overline{E_\Bcal}$ in the Grothendieck category $\Acal_\Bcal$. Let $r_\Bcal: \Acal_\Bcal \to \Acal$ denote the right adjoint to $q_\Bcal$, then $r_\Bcal$ is fully faithful, left exact, and preserves injectives. The object $\widehat{E_\Bcal} = r_\Bcal(\overline{E_\Bcal})$ is injective in $\Acal$, and therefore there is an (up to isomorphism unique, pure-injective) object $E_\Bcal \in \Tcal$ such that $\yo E_\Bcal \cong \widehat{E_\Bcal}$, this is the \newterm{homological residue field object} at $\Bcal$. The \newterm{homological support} of an object $X \in \Tcal$ is defined as $\supph(X) = \{\Bcal \in \Spech{\Tcal^\cpt} \mid [X,E_\Bcal] \neq 0\}$. The topology on $\Spech{\Tcal^\cpt}$ has a basis of closed sets of the form $\{\Bcal \in \Spech{\Tcal^\cpt} \mid [x,E_\Bcal] \neq 0 \}$ for all $x \in \Tcal^\cpt$. Importantly, if $X$ is a weak ring object in $\Tcal$, for example if $X = \lambda_\Dcal$ is a right $\otimes$-idempotent, then by \cite[Theorem 1.8]{Bal20b} we have the more familiar formula $\supph(X) = \{\Bcal \in \Spech{\Tcal^\cpt} \mid \yo_\Bcal(X) \neq 0\} = \{\Bcal \in \Spech{\Tcal^\cpt} \mid E_\Bcal \otimes X \neq 0\}$ (the latter equality holds for any object $X$ by \cite[Proposition 2.14(b)]{Bal20b}).

\subsection{Giraud subcategory} The essential image of $r_\Bcal$ in $\Acal$ is the associated \newterm{Giraud subcategory} $\Gcal_\Bcal$ of $\Acal$. It consists precisely of those objects $M \in \Acal$ such that the unit morphism $M \to r_\Bcal q_\Bcal(M)$ is an isomorphism, or equivalently, such that $\Hom_\Acal(B,M) = 0 = \Ext_\Acal^1(B,M)$ for any $B \in \Loc{\Bcal}$, see \cite[Lemma 2.2]{Kra97}. It follows from \cite[Lemma 2.4, Theorem 2.8]{Kra97} that $r_\Bcal$ preserves direct limits and thus $\Gcal_\Bcal$ is closed under direct limits. In general, $\Gcal_\Bcal$ has no reason to be closed under the tensor product in $\Acal$, but the following special case will prove useful in a moment.
\begin{lemma}\label{Giraud}
  The Giraud subcategory $\Gcal_\Bcal$ is closed under $\yo X \otimes -$ for any $X \in \Tcal$.
\end{lemma}
\begin{proof}
  Since the functor $\yo X$ is flat, it is isomorphic in $\Acal$ to a direct limit of representable functors, see \cite{OR70}. Therefore, it suffices to prove the claim for $x \in \Tcal^\cpt$ in place of general $X \in \Tcal$ by the preceding discussion. By \cite[Proposition 2.9(b)]{Bal20b}, $\yo x$ is rigid with respect to the closed monoidal structure on $\Acal$, so that there is a natural isomorphism $\yo x \otimes - \cong [(\yo x)^*,-]_\Acal$, where $(\yo x)^* = [\yo x,\yo 1]_\Acal$. Then for any $B \in \Loc{\Bcal}$ and $M \in \Gcal_\Bcal$ we have 
  $$\Hom_\Acal(B,\yo x \otimes M) \cong \Hom_\Acal(B,[(\yo x)^*,M]_\Acal) \cong \Hom_\Acal(B \otimes (\yo x)^*,M) = 0$$ 
  using the adjunction and that $\Loc{\Bcal}$ is closed under tensoring. The same argument applies when $\Hom_\Acal$ is replaced by $\Ext_\Acal^1$, which proves $\yo x \otimes M \in \Gcal_\Bcal$. This works because $\yo x \otimes -$ is an exact functor by \cite[Proposition 2.9(b)]{Bal20b}, and the same applies to $(\yo x)^* \otimes - $ once we check that $(\yo x)^* \cong \yo(x^*)$. The latter natural isomorphism is obtained similarly to \cite[Recollection 2.4]{Bal20b}: Indeed, for any $c \in \Tcal^\cpt$ and $Y \in \Tcal$ we have $\Hom_\Acal(\yo c, [\yo x, \yo Y]_\Acal) \cong \Hom_\Acal(\yo(c \otimes x),\yo Y) \cong \Hom_\Tcal(c \otimes x, Y) \cong \Hom_\Tcal(c,[x,Y]) \cong \Hom_\Acal(\yo c, \yo [x,Y])$, establishing the natural isomorphism $\yo[x,Y] \cong [\yo x,\yo Y]_\Acal$ for any $x \in \Tcal^\cpt$ and $Y \in \Tcal$.
\end{proof}
\begin{cor}\label{tensor-formula-Giraud}
  Let $X,Y \in \Tcal$ be such that $\yo Y \in \Gcal_\Bcal$. Then there is a natural isomorphism $\yo (Y \otimes X) \cong \yo Y \otimes r_\Bcal \yo_\Bcal X$. 
  
  In particular, we have $\yo (E_\Bcal \otimes X) \cong \widehat{E_\Bcal} \otimes r_\Bcal \yo_\Bcal X$ for any $X \in \Tcal$.
\end{cor}
\begin{proof}
  By the general theory, both the kernel and the cokernel of the reflection map $i: \yo X \to r_\Bcal q_\Bcal (\yo X)$ in $\Acal$ belong to $\Loc{\Bcal}$. Since $\yo Y$ is $\otimes$-flat by \cite[Proposition 2.9(b)]{Bal20b}, the same is true for the kernel $K$ and cokernel $C$ of the map $\yo Y \otimes i$. On the other hand, the map $\yo Y \otimes i: \yo Y \otimes \yo X \to \yo Y \otimes r_\Bcal q_\Bcal(\yo X)$ belongs to the Giraud subcategory $\Gcal_\Bcal$ by \cref{Giraud} because $\yo Y, r_\Bcal q_\Bcal (\yo X) \in \Gcal_\Bcal$. Since $\yo Y \otimes \yo X \in \Gcal_\Bcal$, the kernel $K$ of $\yo Y \otimes i$ satisfies $\Hom_\Acal(B,K) = 0$ for any $B \in \Loc{\Bcal}$, which together with $K \in \Loc{\Bcal}$ implies that $K$ is zero. Then $\yo Y \otimes i$ is a monomorphism between objects in $\Gcal_\Bcal$, and a long exact sequence argument shows that $\Hom_\Acal(B,C) = 0$ for any $B \in \Loc{\Bcal}$, implying as above that also the cokernel $C$ is zero. Thus, $\yo Y \otimes i$ is the desired isomorphism.

  \noindent
  The last claim follows since $\yo E_\Bcal = \widehat{E_\Bcal} = r_\Bcal (\overline{E_\Bcal}) \in \Gcal_\Bcal$.
\end{proof}
\subsection{``Nerves of Steel'' Conjecture} Following \cite{Bal20}, there is a continuous surjective map $\varphi: \Spech{\Tcal^\cpt} \to \Spec{\Tcal^\cpt}$ given by $\varphi(\Bcal) = \yo^{-1}(\Bcal)$. It is an open question of Balmer \cite[Remark 5.15]{Bal20}, sometimes referred to as the \newterm{``Nerves of Steel'' Conjecture $\NoSC$}, whether this map is also injective in general. In all the examples of big tt-categories coming from algebraic geometry, algebraic topology, and modular representation theory this has been checked to be true \cite{Bal20}. See also the treatments of $\NoSC$ in \cite{BDS23}, \cite{BW23}.

\subsection{$\TC$ via homological residue fields} 
Let us record first the following consequence of the rigidity assumption. Recall that a morphism $f$ in $\Tcal$ is \newterm{nilpotent} if there is $n>0$ such that $f^{\otimes n} = \underbrace{\text{f} \otimes \text{f} \otimes \ldots \otimes \text{f}}_{n \text{ times}}$ is zero.
\begin{lemma}\label{tensor-trick}
  Let $\Dcal$ be a definable $\otimes$-ideal of $\Tcal$ and $x \in \Tcal^\cpt$. If the reflection map $i: x \to x \otimes \lambda_\Dcal$ is nilpotent then $x \otimes \lambda_\Dcal = 0$.
\end{lemma}
\begin{proof}
  The map $i^{\otimes n}: x^{\otimes n} \to (x \otimes \lambda_\Dcal)^{\otimes n} = x^{\otimes n} \otimes \lambda_\Dcal^{\otimes n} = x^{\otimes n} \otimes \lambda_\Dcal$ identifies with the reflection  $x^{\otimes n} \otimes (1 \to \lambda_\Dcal)$ of $x^{\otimes n}$. Therefore, $i^{\otimes n} = 0$ if and only if $x^{\otimes n} \otimes \lambda_\Dcal = 0$ if and only if $x^{\otimes n} \in \Lcal = \Perp{}\Dcal$ if and only if $x \in \Lcal$, the last implication follows from \cite[Proposition 2.4]{Bal07}.
\end{proof}
Recall that a subset $V$ of a topological space is \newterm{specialization closed} if the closure of any point in $V$ is a subset of $V$, or equivalently, if $V$ can be written as a union of closed sets. Any Thomason subset of $\Spec{\Tcal^\cpt}$ is specialization closed but the converse is not true in general.
\begin{lemma}\label{Ep}
  Let $\Dcal$ be a definable $\otimes$-ideal of $\Tcal$. Then:
  \begin{enumerate}
    \item[(1)] For any $\Bcal \in \Spech{\Tcal^\cpt}$, $E_\Bcal \in \Dcal$ if and only if $\Bcal \in \supph(\lambda_\Dcal)$.
    \item[(2)] If $V = \varphi(\supph(\lambda_\Dcal))^c$ is specialization closed then it is Thomason and $\Dcal \subseteq \Tcal_V$.
    \item[(3)] For any $\Bcal \in \Spech{\Tcal^\cpt}$ and a Thomason subset $V$ of $\Spec{\Tcal^\cpt}$ we have $E_\Bcal \in \Tcal_V$ if and only if $\varphi(\Bcal) \not\in V$.
  \end{enumerate}
\end{lemma}
\begin{proof}
  $(1)$: If $\Bcal \in \supph(\lambda_\Dcal)$ then $0 \neq E_\Bcal \otimes \lambda_\Dcal \in \Dcal$. By \cite[Lemma 3.5]{Bal20b}, the morphism $\yo_\Bcal(1 \to \lambda_\Dcal)$ is either monic or $\yo_\Bcal(\lambda_\Dcal) = 0$. Since $E_\Bcal \otimes \lambda_\Dcal \neq 0$, the latter cannot be the case and thus $\yo_\Bcal(1 \to \lambda_\Dcal)$ is monic, and then so is $r_\Bcal \yo_\Bcal(1 \to \lambda_\Dcal)$ in $\Acal$. Since $\widehat{ E_\Bcal} = \yo E_\Bcal$ is $\otimes$-flat in $\Acal$ by \cite[Proposition 2.9(b)]{Bal20b}, also the map $\widehat{ E_\Bcal} \otimes r_\Bcal\yo_\Bcal(1 \to \lambda_\Dcal)$ is monic in $\Acal$. By \cref{tensor-formula-Giraud}, the latter map is identified in $\Acal$ with $\yo(E_\Bcal \to E_\Bcal \otimes \lambda_\Dcal)$. Then, by definition, $E_\Bcal \to E_\Bcal \otimes \lambda_\Dcal$ is a pure monomorphism in $\Tcal$. Since $\Dcal$ is a definable $\otimes$-ideal of $\Tcal$, and thus closed under pure monomorphisms, we conclude that $E_\Bcal \in \Dcal$. On the other hand, if $E_\Bcal \in \Dcal$ then $E_\Bcal \cong E_\Bcal \otimes \lambda_\Dcal$, and so $\Bcal \in \supph(\lambda_\Dcal)$.

  \noindent
  $(2)$: By the assumption, $V$ can be written as a union $\bigcup_{i \in I}V_i$ of closed subsets of $\Spec{\Tcal^\cpt}$. Since $\yo_\Bcal(\lambda_\Dcal) = 0$ for any $\Bcal$ with $\varphi(\Bcal) \in V$, we have by the Tensor Nilpotence Theorem \cite[Corollary 4.7]{Bal20} that there is for each $i \in I$ a Thomason subset $W_i$ containing $V_i$ such that for each $x \in \Kcal_{W_i}$, the reflection map $x \to (x \otimes \lambda_\Dcal)$ is nilpotent, which implies $x \otimes \lambda_\Dcal = 0$ by \cref{tensor-trick}. Put $W = \bigcup_{i \in I}W_i$, this is a Thomason subset of $\Spec{\Tcal^\cpt}$ which contains $V$. We claim that $V=W$. Indeed, otherwise there is $i \in I$, $\PP \in W_i \setminus V$, and $\Bcal \in \Spech{\Tcal^\cpt}$ with $\varphi(\Bcal) = \PP$ such that $\Bcal \in \supph(\lambda_\Dcal)$. Then $E_\Bcal \otimes \lambda_\Dcal = E_\Bcal$ by (1), and so $(x \otimes \lambda_\Dcal) \otimes E_\Bcal = x \otimes E_\Bcal$ cannot be zero for any $x \in \Tcal^\cpt$ with $\PP \in \supp(x) \subseteq W_i$; note that such an $x$ exists by \cite[Proposition 2.14(b)]{Bal05} and that $\supph(x) = \{\Bcal \in \Spech{\Tcal^\cpt} \mid E_\Bcal \otimes x \neq 0\} = \varphi^{-1} \supp(x)$ by \cite[Proposition 4.4]{Bal20b}. This is a contradiction with $x \otimes \lambda_\Dcal = 0$, and thus $V = W$ is Thomason.

  \noindent
  The next claim is that $\Dcal \subseteq \Tcal_V$. Indeed, let $x \in \Kcal_V$ and consider the reflection map $i: x \to x \otimes \lambda_\Dcal$. Since $\supph(x) \subseteq \varphi^{-1}V$ and $\supph(\lambda_\Dcal) \cap \varphi^{-1}V = \emptyset$, we have $\yo_\Bcal(i) = 0$ for all $\Bcal \in \Spech{\Tcal^\cpt}$. Employing Tensor Nilpotence Theorem \cite[Corollary 4.7]{Bal20} again, we have that $i$ is nilpotent, which implies that $i$ is zero by \cref{tensor-trick}. This establishes $\Dcal \subseteq \Kcal_V\Perp{} = \Tcal_V$.
  
  \noindent
  $(3)$: This follows from \cite[Lemma 3.8]{BDS23} and (1). More directly, we have for any $x \in \Tcal^\cpt$ and $\Bcal \in \Spech{\Tcal^\cpt}$ that $[x,E_\Bcal] = 0$ if and only if $\varphi(\Bcal) \not\in \supp(x)$ by \cite[Proposition 4.4]{Bal20b}. On the other hand, $E_\Bcal \in \Tcal_V$ if and only if $\Hom_\Tcal(x,E_\Bcal) = 0$ for any $x \in \Kcal_V$. Since $\Hom_\Tcal(y \otimes x,E_\Bcal) \cong \Hom_\Tcal(y,[x,E_\Bcal])$ for any $y \in \Tcal^\cpt$ and $\Kcal_V$ is closed under tensoring, we see that the latter vanishing condition is further equivalent to $[x,E_\Bcal] = 0$ for any $x \in \Kcal_V$, and thus to $\varphi(\Bcal) \not\in V$.
\end{proof}

\begin{rmk}
  The following alternative slick proof of \cref{Ep}(1) using the result \cite[Proposition 4.4]{BW23} was suggested by the anonymous referee of this paper. The non-zero object $E_\Bcal \otimes \lambda_\Dcal$ belongs to the smallest definable tensor-closed (but not thick!) subcategory $\DefBW{E_\Bcal}$ of $\Tcal$ containing $E_\Bcal$. Then \cite[Proposition 4.4]{BW23} implies that $\DefBW{E_\Bcal} = \DefBW{E_\Bcal \otimes \lambda_\Dcal}$. As $E_\Bcal \otimes \lambda_\Dcal \in \Dcal$ and $\Dcal$ is a definable $\otimes$-ideal, we conclude that $E_\Bcal \in \DefBW{E_\Bcal \otimes \lambda_\Dcal} \subseteq \Dcal$.
\end{rmk}
Let $V$ be a Thomason subset of $\Spec{\Tcal^\cpt}$. Then the closure of any $\PP \in \Spec{\Tcal_V^\cpt}$ contains a closed point of $\Spec{\Tcal_V^\cpt}$ by \cite[Corollary 2.12]{Bal05}. Let us denote by $\clp(V^c)$ the set of points in $V^c$ which are closed in the topology induced by $V^c \cong \Spec{\Tcal_V^\cpt}$. It follows that any point in $V^c$ contains a point from $\clp(V^c)$ in its closure in $\Spec{\Tcal^\cpt}$.

\begin{lemma}\label{cg-general}
  Let $\Tcal$ be a big tt-category and $V$ a Thomason subset of $\Spec{\Tcal^\cpt}$. The following are equivalent:
  \begin{enumerate}
    \item[(i)] every definable $\otimes$-ideal $\Dcal$ such that $\varphi^{-1}(\clp(V^c)) \subseteq \supph(\lambda_\Dcal)$ and $\Dcal \subseteq \Tcal_V$ is compactly generated,
    \item[(ii)] $\Deft{\coprod_{\varphi(\Bcal) \in \clp(V^c)} E_\Bcal }$ is compactly generated,
    \item[(iii)] $\Deft{\coprod_{\varphi(\Bcal) \in \clp(V^c)}E_\Bcal} = \Tcal_V$. 
  \end{enumerate}
\end{lemma}
\begin{proof}
  $(i) \implies (ii):$  Using \cref{Ep}(1) we see that $\Dcal = \Deft{\coprod_{\varphi(\Bcal) \in \clp(V^c)} E_\Bcal }$ satisfies $\varphi^{-1}(\clp(V^c)) \subseteq \supph(\lambda_\Dcal)$. By \cref{Ep}(3), $E_\Bcal \in \Tcal_V$ for all $\Bcal$ such that $\varphi(\Bcal) \in \clp(V^c)$, and so $\Dcal \subseteq \Tcal_V$. Then $\Dcal$ is compactly generated by the assumption $(i)$.

  \noindent
  $(ii) \implies (iii):$ By $(ii)$ and \cref{ttftriples} and \cref{Balm-cg}, there is a Thomason set $W$ such that $\Deft{\coprod_{\varphi(\Bcal) \in \clp(V^c)}E_\Bcal } = \Tcal_W$. By \cref{Ep}(3), $E_\Bcal \in \Tcal_V$ if and only if $\varphi(\Bcal) \not\in V$ for any $\Bcal \in \Spech{\Tcal^\cpt}$. This immediately implies that $\Tcal_W \subseteq \Tcal_V$ and thus $V \subseteq W$. On the other hand, we have $E_\Bcal \in \Tcal_W$ whenever $\varphi(\Bcal) \in \clp(V^c)$. If $V \subsetneq W$, then there is $\Bcal \in \Spech{\Tcal^\cpt}$ such that $\varphi(\Bcal) \in W \cap \clp(V^c)$ (see the paragraph above), which is a contradiction to \cref{Ep}(3) because $E_\Bcal \in \Tcal_W$.
  
  \noindent
  $(iii) \implies (i):$ By ~\cref{Ep}(1), $E_\Bcal \in \Dcal$ for each $\Bcal \in \clp(V^c)$. It follows that $\Deft{\coprod_{\varphi(\Bcal) \in \clp(V^c)}E_\Bcal } \subseteq \Dcal \subseteq \Tcal_V$, which by $(iii)$ implies that $\Dcal = \Tcal_V$ is compactly generated.
\end{proof}
We are almost ready to state a general characterization of $\TC$ in terms of definable $\otimes$-ideals generated by the homological residue field objects, but we need some notation first. Given $\PP \in \Spec{\Tcal^\cpt}$, let $G(\PP) = \{\QQ \in \Spec{\Tcal^\cpt} \mid \PP \subseteq \QQ\}$ be the set of generalizations of $\PP$. Note that the complement $G(\PP)^c$ in $\Spec{\Tcal^\cpt}$ is the Thomason subset $\bigcup_{x \in \PP}\supp(x)$ corresponding to the thick $\otimes$-ideal $\PP$, so that, in the notation of \cref{Balm-cg}, $\Lcal_{G(\PP)^c} = \Loc{\PP}$ is the localizing $\otimes$-ideal generated by $\PP$. For brevity, let us denote $\Tcal_\PP := \Tcal_{G(\PP)^c}$ the \newterm{stalk} tt-category at $\pp$ and $\lambda_\PP := \lambda_{G(\PP)^c}$ the corresponding right $\otimes$-idempotent. Note that $\Tcal_\PP^\cpt$ is a \newterm{local} tt-category in the sense of \cite[\S 4]{Bal10}, and thus $\clp(G(\PP)) = \{\PP\}$.

If $\Tcal$ satisfies $\NoSC$, the Balmer spectrum and the homological spectrum are identified via $\varphi$. Therefore, in such a situation we may and will index the homological residue objects as $E_\pp$ using the Balmer primes $\pp \in \Spec{\Tcal^\cpt}$.
\begin{theorem}\label{tc-general}
  Let $\Tcal$ be a big tt-category which satisfies $\NoSC$. The following are equivalent:
  \begin{enumerate}
    \item[(i)] $\Tcal$ satisfies $\TC$,
    \item[(ii)] for each Thomason subset $V$ of $\Spec{\Tcal^\cpt}$, the definable $\otimes$-ideal
    $\Deft{\coprod_{\PP \in \clp(V^c)}E_\PP}$ is compactly generated,
    \item[(iii)] for each Thomason subset $V$ of $\Spec{\Tcal^\cpt}$, we have 
    $$\Deft{\coprod_{\PP \in \clp(V^c)}E_\PP} = \Tcal_V.$$
  \end{enumerate}
\end{theorem}
\begin{proof}
  $(i) \implies (ii):$ Clear.

  \noindent
  $(ii) \implies (iii):$ Follows directly from the implication $(ii) \implies (iii)$ of \cref{cg-general}.

  \noindent
  $(iii) \implies (i):$ Let $\Dcal$ be a definable $\otimes$-ideal in $\Tcal$ and let us show that it is compactly generated. Put $V = \supph(\lambda_\Dcal)^c$ and recall that $V = \{\PP \in \Spec{\Tcal^\cpt} \mid E_\PP \otimes \lambda_\Dcal =0\}$. We claim that $V$ is a specialization closed set. Indeed, let $\pp \not\in V$ so that $E_\PP \otimes \lambda_\Dcal \neq 0$. Now apply the assumption (iii) for the Thomason set $G(\PP)^c$. Since $\clp(G(\PP)) = \{\PP\}$, we obtain $\Deft{E_\PP} = \Tcal_{\PP}$. We have $E_\PP \in \Dcal$ by \cref{Ep}(1), which implies that $\Tcal_{\PP}  = \Deft{E_\PP} \subseteq \Dcal$. It follows that $E_\QQ \in \Dcal$ for any $\QQ \in G(\PP)$. By \cref{Ep}(1), $E_\QQ \otimes \lambda_\Dcal \neq 0$ for any $\QQ \in G(\PP)$, proving the claim. Now \cref{Ep}(2) applies and shows that $V$ is a Thomason subset of $\Spec{\Tcal^\cpt}$ and that $\Dcal \subseteq \Tcal_V$. Since $\supph(\lambda_\Dcal) = V^c$, we have $\Deft{\coprod_{\PP \in \clp(V^c)}E_\PP} \subseteq \Dcal$ by \cref{Ep}(1). Then $\Dcal = \Tcal_V$ by the implication $(i) \implies (iii)$ of \cref{cg-general} and the assumption $(iii)$ applied for $V$.
\end{proof}

\section{Stalk-locality of Telescope Conjecture}\label{S:three}
In this section we consider locality principles which allow to check $\TC$ locally on covers of $\Spec{\Tcal^\cpt}$. The essential result in this direction is the affine-locality of Balmer and Favi \cite{BF11} which applies to finite covers of $\Spec{\Tcal^\cpt}$ by quasi-compact open sets; here the nomenclature comes from the usual case of the cover of a quasi-compact scheme by open affine subsets. This locality result does not depend on any noetherian assumptions on $\Spec{\Tcal^\cpt}$ or any good behavior of the support theory of $\Tcal$ and has applications even when $\TC$ ends up failing in $\Tcal$ (see \cref{Zariski-loc} below). We remark that this makes this setting rather different from another extremely fruitful approach to locality of $\TC$: The stratification machinery of Benson, Iyengar, and Krause \cite{BIK11}, see also Barthel, Heard, and Sanders for a more recent tensor triangulated approach \cite{BHS23}. In fact, in our examples, it can easily happen that $\TC$ holds but there is no classification of (non-smashing) localizing $\otimes$-ideals in terms of subsets of $\Spec{\Tcal^\cpt}$, see \cite{DP08}, cf. \cref{ex:trunc}. Our goal here is to study a stronger kind of locality with respect to the cover of $\Spec{\Tcal^\cpt}$ by the spectra of stalks $\Spec{\Tcal_\PP^\cpt}$ at all (closed) points, as motivated by the recent result for schemes from \cite{HHZ21}.
\subsection{Affine-locality of $\TC$}
Balmer and Favi \cite{BF11} proved that $\TC$ is an \newterm{affine-local} condition on $\Spec{\Tcal^\cpt}$, meaning that it holds in $\Tcal$ if and only if it holds in $\Tcal_{U_i^c}$ where $\Spec{\Tcal^\cpt} = \bigcup_{i=1}^n U_i$ is a cover by quasi-compact open subsets $U_i$. Actually a slightly stronger statement can be extracted from \cite{BF11}: Given a fixed definable $\otimes$-ideal, its compact generation is affine-local. 
\begin{prop}\label{Zariski-loc}
  Let $\Tcal$ be a big tt-category and $\Dcal$ a definable $\otimes$-ideal in $\Tcal$. Let $\Spec{\Tcal^\cpt} = \bigcup_{i=1}^n U_i$ be a cover by quasi-compact open sets. Then $\Dcal$ is compactly generated in $\Tcal$ if and only if $\Dcal \cap \Tcal_{U_i^c}$ is compactly generated in $\Tcal_{U_i^c}$ for each $i=1,\ldots,n$.
\end{prop}
\begin{proof}
  Denote $\Tcal_i := \Tcal_{U_i^c}$. The left to right implication follows readily, as if $\Dcal = \Scal\Perp{}$ for some subcategory $\Scal$ of $\Tcal^\cpt$ then $\Dcal \cap \Tcal_i = \{S \otimes \lambda_{U_i^c} \mid S \in \Scal\}\Perp{}$ for each $i=1,\ldots,n$, noting that $S \otimes \lambda_{U_i^c}$ is a compact object in $\Tcal_i$ whenever $S \in \Tcal^\cpt$.

  \noindent
  For the converse implication, let $V_i$ be the Thomason subset of $U_i \cong \Spec{\Tcal_i^\cpt}$ such that $\Dcal \cap \Tcal_i = ({\Tcal_i})_{V_i}$, see \cref{Balm-cg}. Consider each $V_i$ as a subset of $\Spec{\Tcal^\cpt}$ and put $V = \bigcup_{i=1}^n V_i$. Then $V$ is a Thomason subset of $\Spec{\Tcal^\cpt}$ by the argument in the proof of \cite[Theorem 6.6]{BF11}. Alternatively, one can argue using \cref{Ep}(2) by noticing that first that $V$ is clearly specialization closed in $\Spec{\Tcal^\cpt}$ and that $V = \varphi(\supph{\lambda_\Dcal})^c$ using \cref{Ep}(1),(3) and the fact that $\supph$ can be computed locally. \cref{Ep}(2) also yields $\Dcal \subseteq \Tcal_V$. It remains to show that $\lambda_V \in \Dcal$ so that $\Dcal = \Tcal_V$. By passing from $\Tcal$ to $\Tcal_V$, the statement reduces to the case of $V = \emptyset$. By induction, we can assume that we have a cover $\Spec{\Tcal^\cpt} = U_1 \cup U_2$ by just two quasi-compact opens. Using \cite[Theorem 5.18]{BF11}, there is a triangle in $\Tcal$ of the form $1 \to \lambda_{U_1^\cpt} \oplus \lambda_{U_2^\cpt} \to \lambda_{(U_1 \cap U_2)^\cpt} \mathrel{\leadsto}$. Since $\Tcal_i \subseteq \Dcal$ for $i=1,2$, the second and third term of the triangle belongs to $\Dcal$, and so $1 \in \Dcal$.
\end{proof}
\subsection{Stalk-locality of $\TC$}\label{ss-stalkloc} It was shown in \cite[Theorem 4.9]{HHZ21} that in the case of the derived category of a quasi-compact and quasi-separated scheme $X$, a stronger locality principle holds: $\TC$ holds in $\D(X)$ whenever it holds in the stalk tt-category $\D(O_x)$ over each (closed) point $x \in X$. As in the case of affine-locality, it also holds that the compact generation of a given definable $\otimes$-ideal can be checked stalk-locally in this situation \cite[Proposition 3.13]{HHZ21}.

Motivated by this, we can formulate a similar condition in $\Tcal$. We say that a big tt-category satisfies the \newterm{Stalk-Locality Principle $\SLP$} if any definable $\otimes$-ideal $\Dcal$ of $\Tcal$ is compactly generated whenever $\Dcal \cap \Tcal_\PP$ is compactly generated in $\Tcal_\PP$ for any point $\PP \in \Spec{\Tcal^\cpt}$. In fact, it is sufficient to check the latter condition just on the closed points of $\Spec{\Tcal^\cpt}$, as $\Dcal \cap \Tcal_\PP$ being compactly generated in $\Tcal_\PP$ implies that $\Dcal \cap \Tcal_\QQ$ is compactly generated in $\Tcal_\QQ$ whenever $\PP$ is in the closure of $\QQ$ (cf. \cite[Proposition 4.4]{BF11} and the paragraph before \cref{cg-general}). This is a stronger locality property than affine-locality, but it is not clear to the author if it is satisfied by all big tt-categories.  

In our attempt to provide a local version of \cref{tc-general}, it will make sense to consider a stronger condition in which stalk-locality of compact generation holds in all compact localizations of $\Tcal$. The satisfaction of $\TC$ passes from $\Tcal$ to $\Tcal_V$ for any Thomason set $V$, and in fact, to any smashing localization of $\Tcal$, see \cite[Proposition 4.4]{BF11}. However, the same is not so clear for $\SLP$. In the following, we characterize the situation in which $\SLP$ holds in each compact localization $\Tcal_V$ of $\Tcal$ for all Thomason sets $V$.
\begin{prop}\label{cg-stalklocal}
  Let $\Tcal$ be a big tt-category. The following are equivalent:
  \begin{enumerate}
    \item[(i)] a definable $\otimes$-ideal $\Dcal$ of $\Tcal$ is compactly generated if and only if $\Dcal \cap \Tcal_\PP = \Tcal_\PP$ for each $\PP \in \varphi(\supph(\lambda_\Dcal))$,
    \item[(ii)] $\SLP$ holds in $\Tcal_V$ for any Thomason subset $V$ of $\Spec{\Tcal^\cpt}$,
    \item[(iii)] for any Thomason subset $V$ of $\Spec{\Tcal^\cpt}$, the definable $\otimes$-ideal $\Deft{\coprod_{\PP \not\in V}\lambda_{\PP}}$ is compactly generated.
    \item[(iv)] for any Thomason subset $V$ of $\Spec{\Tcal^\cpt}$, we have 
    $$\Deft{\coprod_{\substack{\MM \in \clp(V^c)}}\lambda_\MM } = \Tcal_V.$$
  \end{enumerate}
\end{prop}
\begin{proof}
  $(i) \implies (ii):$ Let $\Dcal$ be a definable $\otimes$-ideal of $\Tcal_V$ and let us show that its compact generation follows from the compact generation of $\Dcal \cap (\Tcal_V)_\PP$ for each point $\PP \in \Spec{\Tcal_V^\cpt} \cong V^c \subseteq \Spec{\Tcal^\cpt}$. Clearly, this is the same as assuming that $\Dcal \cap \Tcal_\PP$ is compactly generated in $\Tcal_\PP$ for each point $\PP \in V^c$, note that $\Dcal$ is also a definable $\otimes$-ideal inside $\Tcal$. Since $\varphi(\supph{\lambda_\Dcal}) \subseteq V^c$ and $\Dcal \cap \Tcal_\PP = \lambda_\PP \otimes \Dcal$ is compactly generated in $\Tcal_\PP$, it follows that $\Dcal \cap \Tcal_\PP = \Tcal_\PP$ for each $\PP \in \varphi(\supph{\lambda_\Dcal})$, using \cref{Ep}(1),(3). Therefore, $\Dcal$ is compactly generated in $\Tcal$ by (i), and thus it is compactly generated also in $\Tcal_V$.

  \noindent
  $(ii) \implies (iii)$:  Let $\Dcal = \Deft{\coprod_{\PP \not\in V}\lambda_{\PP}} \subseteq \Tcal_V$ and consider it as a definable $\otimes$-ideal inside $\Tcal_V$. Clearly, $\Dcal \cap (\Tcal_V)_\PP = (\Tcal_V)_\PP = \Tcal_\PP$ is compactly generated in $\Tcal_V$ for any $\PP \in V^c$. By $(ii)$, $\Dcal$ is compactly generated in $\Tcal_V$.

  \noindent
  $(iii) \implies (iv):$ In view of the discussion before \cref{cg-general}, $\Dcal = \Deft{\coprod_{\substack{\MM \in \clp(V^c)}}\lambda_\MM } = \Deft{\coprod_{\PP \not\in V}\lambda_{\PP}}$, and the latter definable $\otimes$-ideal is compactly generated by $(iii)$, and thus corresponds to some Thomason subset $W$ of $\Spec{\Tcal_V^\cpt} \cong V^c$ via \cref{Balm-cg}. But since $\lambda_\PP \in \Dcal$ for all $\PP \in V^c$, the only possibility is $W = \emptyset$ and $\Dcal = \Tcal_V$.
  \noindent
  $(iv) \implies (i):$ First, if $\Dcal$ is compactly generated in $\Tcal$ then so is $\Dcal \cap \Tcal_\PP = \lambda_\PP \otimes \Dcal$ in $\Tcal_\PP$, so we only need to prove the converse. Put $V = \varphi(\supph(\lambda_\Dcal))^c$ and let us first note that this is a specialization closed set. Indeed, whenever $\PP \in \varphi(\supph(\lambda_\Dcal))$ then the assumption ensures that $\Tcal_{\PP} \subseteq \Dcal$ and so $\QQ \in \varphi(\supph(\lambda_\Dcal))$ whenever $\QQ$ generalizes $\PP$. By \cref{Ep}(2), $V$ is a Thomason subset and $\Dcal \subseteq \Tcal_V$. If $\MM \in \clp(\varphi(\supph(\lambda_\Dcal)))$ then $\lambda_\MM \in \Dcal$ by the assumption. It follows by $(iii)$ that $\Tcal_V \subseteq \Dcal$, and so $\Dcal = \Tcal_V$ is compactly generated.
\end{proof}

If the equivalent conditions of \cref{cg-stalklocal} hold for $\Tcal$ we will say that $\Tcal$ satisfies the \newterm{Hereditary Stalk-Locality Principle $\HSLP$}.

\begin{rmk}
Note that by \cref{Zariski-loc}, the validity of the equivalent conditions of \cref{cg-stalklocal} in $\Tcal$ can be checked on a cover by quasi-compact open sets. 
\end{rmk}

If we found ourselves in the lucky situation in which both the compact generation is stalk-local in every compact localization and the ``Nerves of Steel'' Conjecture holds, we obtain the following neat reformulation of the Telescope Conjecture.
\begin{cor}\label{tc-stalklocal}
  Let $\Tcal$ be a big tt-category satisfying both $\NoSC$ and $\HSLP$. Then $\TC$ holds in $\Tcal$ if and only if for each $\PP \in \Spec{\Tcal^\cpt}$ we have $\Deft{E_\PP} = \Tcal_\PP$.
\end{cor}
\begin{proof}
  The right to left implication follows directly from \cref{tc-general}, using that $\NoSC$ holds for $\Tcal$. Let us prove the left to right implication by employing the criterion of \cref{tc-general} again. Let $V$ be a Thomason subset of $\Spec{\Tcal^\cpt}$, put $\Dcal = \Deft{\coprod_{\PP \in \clp(V^c)}E_\PP}$, and let us check that $\Dcal = \Tcal_V$. By our assumption, $\lambda_\PP \in \Tcal_\PP = \Deft{E_\PP} \subseteq \Dcal$ for any $\PP \not\in V$. By \cref{cg-stalklocal}$(iv)$, $\Deft{\coprod_{\PP \not\in V}\lambda_\PP} = \Tcal_V$, and so we are done.
\end{proof}
\begin{rmk}
  The author is not aware of any example of a big tt-category which fails $\SLP$. In fact, \cref{cg-stalklocal}$(iii)$ shows that a failure of $\SLP$ would yield a rather catastrophic example of failure of $\TC$: A big tt-category $\Tcal$ in which $\Deft{\coprod_{\PP \in \Spec{\Tcal^\cpt}}\lambda_\PP}$ is not compactly generated.
\end{rmk}
\begin{rmk}\label{lattice}
  The equivalence classes of smashing localizations of a big tt-category $\Tcal$ form a complete lattice with the order given in \cite[Proposition 3.11]{BF11}, see \cite[Corollary 4.12]{Kr00} and also \cite{BKS20}. By taking the corresponding smashing $\otimes$-ideals via \cref{ttftriples}, this lattice is a subposet (in the reverse order) of the complete lattice of Bousfield classes in $\Tcal$ of \cite{IK13}. The join in this lattice corresponds to the intersection of the corresponding definable $\otimes$-ideals via \cref{ttftriples}. It follows straightforwardly that given an object $X \in \Tcal$, the definable $\otimes$-ideal $\Deft{X}$ corresponds via \cref{ttftriples} to the largest smashing localization $- \otimes \lambda$ in the lattice for which $X$ is local, that is, the reflection map $X \to X \otimes \lambda$ is an isomorphism. In this way, the criterion of \cref{tc-stalklocal} has the following interpretation: To check $\TC$, it is enough to show that for each Balmer prime $\PP$, the largest smashing localization making $E_\pp$ local is compactly generated, and thus identifies with the localization $\Tcal \to \Tcal_\pp \subseteq \Tcal$ at $\PP$.
\end{rmk}
\begin{ex}\label{exSH}
    Let $p$ be a prime number and $\Tcal = \SH_p$ the stable homotopy category of $p$-local spectra. Let us explain how \cref{tc-stalklocal} compares to the long-known equivalent formulation of $\TC$ in $\SH_p$ in terms of the Morava $K$-theories. We use the survey paper \cite{Bar20} as our base reference here. First, \cref{tc-stalklocal} applies here, as $\SH_p$ satisfies both $\NoSC$ by \cite[Corollary 5.10]{Bal20} and $\HSLP$ as any compact localization of $\SH_p$ is a local tt-category, this follows from $\Spec{\SH_p^\cpt}$ being homeomorphic to the ordered set $\{0,1,2,\ldots,\infty\}$, \cite[Corollary 9.5]{Bal10}.

    For each $0 \leq n \leq \infty$ interpreted as a point of $\Spec{\SH_p^\cpt}$, the homological residue field object at $n$ is the Morava $K$-theory $K(n)$, see \cite[Lemma 3.5]{BC20}. Here by convention, $K(0)$ and $K(\infty)$ are the Eilenberg-MacLane spectra of $\Qbb$ and $\Zbb/p\Zbb$, respectively. Let $L_n: \Tcal \to \Tcal$ be the Bousfield localization whose kernel is the Bousfield class $\Ker(\coprod_{i=0}^n K(n) \otimes -)$. If $n \neq \infty$ then $L_n$ is a smashing localization by the result of Hopkins and Ravenel so that $L_n \cong (L_n(1) \otimes -)$, see \cite[Theorem 3.3]{Bar20}. Then it is known that $\TC$ in $\SH_p$ is equivalent to all the smashing localizations $L_n$ for $0 \leq n < \infty$ being compactly generated, see \cite[Corollary 4.5]{Bar20}.

    To compare this with \cref{tc-stalklocal}, we first claim that if $n \neq \infty$ then $L_n$ is precisely the largest smashing localization $L = - \otimes \lambda$ for which $K(n)$ is local (see \cref{lattice}). As $K(n)$ is $L_n$-local, $L_n \leq L$ in the lattice of smashing localizations. To see the converse relation, it is enough to check that $K(m)$ is $L$-local for all $m<n$. By \cite[Lemma 4.1]{Bar20}, we have $L(K(m)) = K(m) \otimes \lambda \neq 0$ for any $m<n$. By \cref{Ep}(1), $K(m) \to K(m) \otimes \lambda$ is an isomorphism, as desired. The case $n = \infty$ behaves differently. Indeed, the largest smashing localization for which $K(\infty)$ is local is just the identity functor, this follows from \cite[Proposition 4.4]{Bar20}. On the other hand, there are non-zero objects which are $L_\infty$-acyclic (in terms of the next subsection, this reflects that the homological support $\supph$ does not detect vanishing in $\SH_p$), see \cite[Corollary B.12]{HS99}, and so $L_\infty$ is not smashing.

    It follows that the recently proved failure of $\TC$ in $\SH_p$ of \cite[Theorem A]{BHLS23} can be stated in view of \cref{tc-stalklocal} and \cref{lattice} as follows: For each $1 < n < \infty$, the definable $\otimes$-ideal $\Deft{K(n)}$ is not compactly generated.

    In view of \cref{exBFS}, one can give a similar interpretation in terms of \cref{tc-stalklocal} of the failure of $\TC$ in the non-localized stable homotopy category $\SH$ of spectra, we leave the details up to the reader.
\end{ex}
\begin{ex}
   Let $R$ be a local and zero-dimensional commutative ring, so that $\Spec{\D(R)^\cpt} \cong \Spec{R}$ is a single point set, see \cref{S:four} for notation and references. Assume further that there is a non-zero object $X \in \D(R)$ with $\supph(X) = \emptyset$ and that $\D(R)$ satisfies $\TC$, such examples exist by \cref{ex:trunc} and \cite[Example 5.5]{BDS23}). Then the largest smashing localization for which the unique residue field $k$ is local has to be the identity functor on $\D(R)$ by \cref{tc-stalklocal}, while the Bousfield class $\Ker(k \otimes_R^\mathbf{L} -)$ is non-trivial. This shows that in general, there is no direct analogue to the Hopkins-Ravenel smashing theorem used in \cref{exSH}.
\end{ex}
\subsection{Stalk-locality and support theories} We conclude this section by showing that a weak notion of support vanishing detection implies $\HSLP$. However, as explained in \cref{ex-commalg} below, it is not clear if this argument applies even to the commutative algebra setting of the further sections. For $X \in \Tcal$, we define its \newterm{Na\"{i}ve support} 
$$\SuppN(X) = \{\PP \in \Spec{\Tcal^\cpt} \mid X \otimes \lambda_\PP \neq 0\}.$$
Note that $\SuppN(X)$ is always a specialization closed subset of $\Spec{\Tcal^\cpt}$, a property which makes it very different from the behavior expected of the support theories like $\supph$ or the Balmer-Favi support \cite{BF11} for a general non-compact object $X$. If $\Tcal = \D(R)$ for a commutative ring $R$ then $\SuppN(X) = \Supp(\bigoplus_{n \in \Zbb}H^n(X))$ is just the ordinary support $\Supp$ of the cohomology modules, so that $\Supp(M) = \{\pp \in \Spec R \mid M \otimes_R R_\pp \neq 0\}$ for an $R$-module $M$. On the other hand, $\supph(X) = \{X \in \D(R) \mid \kappa(\pp) \otimes_R^\mathbf{L} X \neq 0\}$ coincides with the usual cohomological support of complexes.

We say that a given support theory $\Supp_*$ (see \cite[Definition 7.1]{BHS23}) in $\Tcal$ \newterm{detects vanishing} if $\Supp_*(X) = \emptyset$ implies $X = 0$ for any $X \in \Tcal$. 
\begin{lemma}\label{stalkloc-supp}
  If $\SuppN$ detects vanishing in $\Tcal$ then 
  $$\Deft{\coprod_{\substack{\MM \in \clp(\Spec{\Tcal^c})}}\lambda_\MM} = \Tcal.$$
  As a consequence, $\HSLP$ holds in $\Tcal$ provided that $\SuppN$ detects vanishing in $\Tcal_V$ for all Thomason subsets $V$ of $\Spec{\Tcal^\cpt}$.
\end{lemma}
\begin{proof}
  Let $\Dcal = \Deft{\coprod_{\substack{\MM \in \clp(\Spec{\Tcal^c})}}\lambda_\MM}$ and put $\Lcal = \Perp{}\Dcal$. Then any $X \in \Lcal$ satisfies $X \otimes \lambda_\PP = 0$ for any $\PP \in \Spec{\Tcal^\cpt}$ so that $\SuppN(X) = \emptyset$. By the assumption, this implies $\Lcal = 0$ and so $\Dcal = \Tcal$. The last claim follows from \cref{cg-stalklocal}$(iv)$.
\end{proof}
\begin{ex}
  If $\Spec{\Tcal^\cpt}$ is weakly noetherian (see \cite[Definition 2.3]{BHS23}) $\Tcal$ satisfies the Local-To-Global Principle of Benson-Iyengar-Krause (see \cite{BIK11} and also \cite{BHS23}) then $\Tcal$ automatically satisfies $\HSLP$. Indeed, the Local-To-Global Principle passes to $\Tcal_V$ for all Thomason sets $V \subseteq \Spec{\Tcal^\cpt}$, this follows from \cite[Corollary 3.14]{BHS23}. Since the Local-To-Global principle implies that the Balmer-Favi support detects vanishing, see \cite[Corollary 3.10]{BHS23}, the rest follows from \cref{stalkloc-supp}. Indeed, it is easy to see from the definition of the two support theories that for any $X \in \Tcal$ we have $\SuppN(X) = \emptyset$ implies that the Balmer-Favi support of $X$ is empty as well.
\end{ex}
\begin{ex}\label{exBFS}
  More generally, if the Balmer-Favi-Sanders $\SuppBFS$ detects vanishing in $\Tcal$ then it satisfies $\HSLP$. W. T. Sanders generalized in \cite{San17} the Balmer-Favi support of \cite{BF11} to the case when $\Spec{\Tcal^\cpt}$ is not necessarily weakly noetherian, see also the recent work of Zou \cite{Zou23}. It follows directly from the definition that if $\SuppBFS$ detects vanishing in $\Tcal$ then $\SuppN$ detects vanishing in $\Tcal_V$ for each Thomason subset $V$ of $\Spec{\Tcal^\cpt}$, and so \cref{stalkloc-supp} applies. 
  
  In particular, $\HSLP$ holds for the stable homotopy category $\SH$ of spectra, see \cite[Example 6.19]{Zou23}. We remark that it is not known whether $\SuppBFS$ detects vanishing for \textit{any} big tt-category $\Tcal$, or even for $\D(X)$ for a general $X$, see \cite[\S 8.1]{San17} and \cite[Remark 6.10]{Zou23}.
\end{ex}

\begin{ex}\label{ex-commalg}
  It is clear that in the derived category of a commutative ring $\D(R)$, $\SuppN$ detects vanishing and so $\D(R) = \Deft{\coprod_{\mm \in \clp(\Spec R)}R_\mm}$. However, the same condition becomes less apparent in compact localizations of $\D(R)$. Indeed, it is straightforward to check that $\SuppN$ detects vanishing in $\D_V := \D(R)_V$ for a Thomason set $V$ if and only if the smashing subcategory $\Lcal_V$ of $\D(R)$ contains any complex $X$ whose cohomology $H^*(X)$ is supported inside $V$. This is known to hold if $V$ is a complement of a single quasi-compact open set \cite[Theorem 6.8]{Rou08} or if $R$ is commutative noetherian \cite{ATJLS10}, but seems to be unknown in general, cf. the discussion in \cite[\S 3.1]{HHZ21}. The issue lies with complexes which are not bounded from the left, as in principle, there could be such $X$ in $\D_V = \Lcal_V\Perp{}$ with $\SuppN(X) = \Supp(H^*(X)) \subseteq V$. Nevertheless, it turns out that $\D(R)$ always satisfies $\HSLP$, as we prove below in \cref{biloc-stalks} by different means. 
\end{ex}

\section{Derived category of a scheme}\label{S:four}
From now on, we specialize to a specific big tt-category. Given a quasi-compact and quasi-separated scheme $X$, let $\D(X)$ denote the derived category of unbounded cochain complexes of $\Ocal_X$-modules with quasi-coherent cohomology. Then $\D(X)$ is a big tt-category with the role of the tensor structure played by the derived tensor product $- \otimes_X^\mathbf{L} -$, see \cite[\S 1.2]{BF11} for the relevant references. In particular, $\D(X)$ is compactly generated with $\D(X)^\cpt$ identified with the full subcategory of perfect complexes, and we have the homeomorphism $\Spec{\D(X)^\cpt} \cong X$ by Thomason's result \cite{Tho97}. 

The locality results discussed in the previous section reduce the study of $\TC$ to the case of an affine scheme $X \cong \Spec{R}$, where $R$ is a commutative or even a local commutative ring. In this case, a subset $V$ of $\Spec R$ is Thomason if and only if $V$ can be written as a union of Zariski closed sets $V(I) = \{\pp \in \Spec R \mid I \subseteq \pp\}$ where $I$ is a finitely generated ideal of $R$. Given a finitely generated ideal $I$, let $K(I) = \bigotimes_{i=1}^n (R \xrightarrow{\cdot x_i} R)$ concentrated in degrees $-n,\ldots,0$ be the Koszul complex defined over some fixed generating set $x_1,\ldots,x_n$ of $I$, then $K(I)$ is a compact object satisfying $\supp(K(I))=V(I)$. Recall that the unit object $R$ generates $\D(R) := \D(\Spec{R})$ as a localizing subcategory. It follows that each thick subcategory of $\D(R)^\cpt$ is automatically a thick $\otimes$-ideal and each localizing subcategory of $\D(R)$ is automatically a localizing $\otimes$-ideal. There is an explicit description of the right $\otimes$-idempotent associated to a Thomason set of the form $V(I)$ with $I = (x_1,\ldots,x_n)$ a finitely generated ideal of $R$. Indeed, $R \to \lambda_{V(I)}$ can be represented by the the mapping cone of the canonical chain map $\bigotimes_{i=1}^n (R \to R[x_i^{-1}]) \to R$, where $\bigotimes_{i=1}^n (R \to R[x_i^{-1}])$ concentrated in degrees $0,1,\ldots,n$ is the \newterm{Čech complex} on $I$, see \cite[\S 5.2]{Hrb20} for details and references.

We will also refresh some notation from the previous section. Given a Thomason subset $V$ of $\Spec R$, we have $\Kcal_V = \{X \in \D(R)^\cpt \mid \supp(X) \subseteq V\} = \thick{\bigcup_{V(I) \subseteq V}K(I)}$, $\Lcal_V = \Loc{\Kcal_V}$, and $\D_V = \Kcal_V\Perp{}$ (the latter replaces the notation $\Tcal_V$ from the previous section). The following recovers and strengthens \cite[Proposition 3.13]{HHZ21} for stable t-structures. 
\begin{theorem}\label{biloc-stalks}
  Let $X$ be a quasi-compact and quasi-separated scheme. Then $\D(X)$ satisfies $\HSLP$.
\end{theorem}
\begin{proof}
  By \cref{Zariski-loc}, we can reduce to the case of an affine scheme $X = \Spec R$. In view of \cref{cg-stalklocal}$(iv)$, given a Thomason subset $V$ of $\Spec R$ we need to check that $\Def{\bigoplus_{\pp \not\in V}R_\pp} = \D_V$. Since $R_\pp \in \D_V$ for each $\pp \not\in V$, we clearly have $\Def{\bigoplus_{\pp \not\in V}R_\pp} \subseteq \D_V$. To prove the other inclusion, it suffices to show that $\lambda_V\in \Def{\bigoplus_{\pp \not\in V}R_\pp}$. By \cref{hocolim}, $\lambda_V$ is a directed homotopy colimit of $\lambda_{V(I)}$ for Zariski closed sets $V(I) \subseteq V$ with $I$ a finitely generated ideal. Since each $\lambda_{V(I)}$ has vanishing cohomology in negative degrees by the discussion above, we also have $H^i(\lambda_V) = 0$ for all $i<0$. We shall show that any object from $\D_V \cap \D^+(R)$ belongs to $\Def{\bigoplus_{\pp \not\in V}R_\pp}$, where $\D^+(R) = \{X \in \D(R) \mid H^i(X) = 0 ~\forall i \ll 0\}$ is the subcategory of objects which are cohomologically bounded from below.

  \noindent
  First, we handle the special case of a stalk complex $W[0] \in \D_V$ of an injective $R$-module $W$. Since $\D_V = \Kcal_V\Perp{}$, we have $\Hom_R(R/I,W)=0$ for any finitely generated ideal $I$ with $V(I) \subseteq V$. We claim that the canonical map $W \to \prod_{\pp \not\in V}W_\pp$ is injective, where $W_\pp = W \otimes_R R_\pp$. Indeed, let $K$ be the kernel of this map, then $K_\pp = 0$ for all $\pp \not\in V$. Assume that $K$ is non-zero and let $R/J$ be a non-zero cyclic submodule of $K$. By \cite[Lemma 3.2]{Hone}, there is a prime $\qq$ such that $R/\qq$ can be constructed from $R/J$ by taking submodules and direct limits. Then $(R/\qq)_\pp = 0$ for all $\pp \not\in V$, showing that $\qq \in V$. Since $V$ is Thomason, there is a finitely generated ideal $I$ such that $V(I) \subseteq V$ and $I \subseteq \qq$. This is in contradiction with $\Hom_R(R/I,W) = 0$. Since $W$ is injective, the monic map $W \to \prod_{\pp \not\in V}W_\pp$ splits. This shows that $W[0] \in \Def{\bigoplus_{\pp \not\in V}R_\pp}$. 
  
  \noindent
  Now we follow the argument of \cite[Lemma 3.4]{Hrb20}. Let $X \in \D_V \cap \D^+(R)$. Then the stalk of the injective envelope $E(H^{\inf(X)}(X))$ of the left-most non-vanishing cohomology of $X$ belongs to $\D_V$ by \cite[Lemma 3.3]{Hrb20}, and therefore also to $\Def{\bigoplus_{\pp \not\in V}R_\pp}$ by the previous paragraph. As in \cite[Lemma 3.4]{Hrb20}, we can iteratively construct an injective resolution of $X$ using injective $R$-modules belonging to $\Def{\bigoplus_{\pp \not\in V}R_\pp}$, and so a standard argument by a Milnor limit of brutal truncations yields $X \in \Def{\bigoplus_{\pp \not\in V}R_\pp}$. 
\end{proof}
Given a point $x \in X$, let $\Ocal_x$ be the stalk local ring and $k(x)$ be the residue field sheaf at $x$. Then \cref{tc-stalklocal} yields the following criterion for the Telescope Conjecture in $\D(X)$.
\begin{theorem}\label{tc-scheme}
  Let $X$ be a quasi-compact and quasi-separated scheme. Then the following are equivalent:
  \begin{enumerate}
    \item[(i)] $\TC$ holds for $\D_{\qc}(X)$,
    \item[(ii)] for each $x \in X$ we have $\Deft{k(x)} = \D(\Ocal_x)$.
  \end{enumerate}
\end{theorem}
\begin{proof}
  This follows from \cref{biloc-stalks} and \cref{tc-stalklocal}. Indeed, $\NoSC$ holds for $\D(R)$ by \cite[Corollary 5.11]{Bal20}, and for each $x \in X$ the corresponding homological residue field object is precisely the usual residue field $k(x)$ viewed as a stalk complex in degree zero, see \cite[Corollary 3.3]{BC20}.
\end{proof}
\section{Telescope conjecture vs. ring epimorphisms}
In the derived category of a (even not necessarily commutative) ring $R$, the smashing localizations of $\D(R)$ are represented by certain dg-ring extensions of $R$ see \cite{NS09}, \cite{Pau09}, \cite[Proposition 2.5]{BS17}. In case of rings with a suitable homological dimension restriction, these can be replaced by epimorphisms of ordinary rings and the Telescope Conjecture can be reformulated in terms of these ring extensions, see \cite{BS17} and \cite{KS10}. In this section, we study those definable $\otimes$-ideals in $\D(R)$ for a commutative ring $R$ which arise in such a way from ordinary ring extensions. We say that a definable $\otimes$-ideal $\Dcal$ of $\D(R)$ is \newterm{closed under cohomology} if $X \in \Dcal$ implies $H^*(X)=\bigoplus_{n \in \Zbb}H^n(X) \in \Dcal$. As a consequence, $X \in \Dcal$ if and only if $H^*(X) \in \Dcal$. Here, the implication $H^*(X) \in \Dcal \implies X \in \Dcal$ holds for any definable $\otimes$-ideal $\Dcal$. Indeed, one can reconstruct any cohomologically bounded complex $X \in \D(R)$ from its cohomology modules using shifts and extensions, and this extends to unbounded complexes using the Milnor limits and colimits of soft truncations; note that $\Dcal$ is closed under all of these operations.

Recall that a ring epimorphism $f: R \to S$ is \newterm{pseudoflat} if $\Tor_1^R(S,S) = 0$ and it is \newterm{flat} if $S$ is flat as an $R$-module, that is, if $\Tor_1^R(S,M) = 0$ for all $R$-modules $M$.
\begin{lemma}\label{epi-biloc}
  For any (even not necessarily commutative) ring $R$, the following collections are in bijection:
  \begin{enumerate}
    \item[(i)] extension-closed bireflective subcategories $\Xcal$ of $\Mod R$,
    \item[(ii)] pseudoflat ring epimorphisms $f: R \to S$ up to equivalence,  
    \item[(iii)] definable $\otimes$-ideals $\Dcal$ of $\D(R)$ closed under cohomology. 
  \end{enumerate}
  The bijection $(i) \leftrightarrow (iii)$ is given as $\Xcal \mapsto \D_\Xcal = \{X \in \D(R) \mid H^*(X) \in \Xcal\}$. The bijection $(i) \leftrightarrow (ii)$ assigns to $f: R \to S$ the image $\Xcal_f$ of the forgetful fully faithful functor $\Mod S \to \Mod R$.
\end{lemma}
\begin{proof}
  The bijection between $(i)$ and $(ii)$ is well-known, see e.g. \cite[Theorem 2.1]{AMSVT} and references there given. Next, the assignment $(i) \to (iii)$ is well-defined and injective. Indeed, $\Xcal$, being a bireflective subcategory, is closed under all kernels and cokernels. Then $\D_\Xcal$ is a thick subcategory of $\D(R)$ by a straightforward argument with the long exact sequence on cohomology. Furthermore, $\D_\Xcal$ is definable as it is closed under both products and pure monomorphisms, which follows from $\Xcal$ being closed under products and pure monomorphisms in $\Mod R$. To show that the assignment is surjective, let $\Dcal$ be a definable $\otimes$-ideal of $\D(R)$ closed under cohomology, and let $\Xcal = \{H^0(X) \mid X \in \Dcal\}$. Since $\Dcal$ is closed under cohomology, we have $\Xcal[n] \subseteq \Dcal$ for all $n \in \Zbb$. Because $\Dcal$ is closed under extensions and both Milnor limits and colimits, a standard argument using truncations shows that $\Dcal = \{X \in \D(R) \mid H^*(X) \in \Xcal\}$. Then $\Xcal = \Dcal \cap \Mod R[0]$ is closed under kernels, cokernels, products, coproducts, and extensions in $\Mod R$, and thus it is an extension-closed bireflective subcategory.
\end{proof}

\begin{rmk}\label{rem-wide}
  A subcategory $\Xcal$ of $\Mod R$ is extension-closed bireflective if and only if it is a wide subcategory (i.e., closed under extensions, kernels, and cokernels) which is closed under products and coproducts, see \cite[Theorem 1.2]{GdP}. This is further equivalent to $\Xcal$ being a definable wide subcategory, in terms of the usual pure exact structure on $\Mod R$. Indeed, if a wide subcategory is definable then it is closed under both products and coproducts, while an extension-closed bireflective subcategory $\Xcal$ is definable because it can be written as an intersection $\Xcal = \D_\Xcal \cap \Mod R[0]$ of two definable subcategories of $\D(R)$ by \cref{epi-biloc} (cf. \cite[\S 17.3]{Prest}).
  
  In this way, we can see how extension-closed bireflective subcategories provide the suitable restriction of the notion of definable thick subcategories (=definable $\otimes$-ideals in this case) from the triangulated category $\D(R)$ to the abelian category $\Mod R$.
\end{rmk}

In view of the lemma, we formulate the following restricted version of $\TC$. 

\begin{dfn}
  We say that $\D(R)$ satisfies the \newterm{Restricted Telescope Conjecture $\RTC$} provided that each definable $\otimes$-ideal $\Dcal$ of $\D(R)$ which is closed under cohomology is compactly generated.
\end{dfn}
\begin{lemma}\label{flat-injenv}
  Let $R$ be a commutative ring and $f: R \to S$ be a pseudoflat ring epimorphism. The following are equivalent:
  \begin{enumerate}
    \item[(i)] $f$ is a flat ring epimorphism,
    \item[(ii)]  $\Mod S$ is closed under injective envelopes as a subcategory of $\Mod R$.
  \end{enumerate}
\end{lemma}
\begin{proof}
  The condition $(i)$ is equivalent to $- \otimes_R S: \Mod R \to \Mod S$ being exact, which is by a standard argument equivalent to its right adjoint, the forgetful functor $\Mod S \to \Mod R$, preserving injective objects. The latter condition is clearly equivalent to $(ii)$.
\end{proof}
\begin{theorem}\label{flat-tc}
  Let $R$ be a commutative ring and $f: R \to S$ a pseudoflat epimorphism. The following are equivalent:
  \begin{enumerate}
    \item[(i)] $f: R \to S$ is flat,
    \item[(ii)] the corresponding definable $\otimes$-ideal $\D_{\Xcal_f}$ is compactly generated. 
  \end{enumerate}

  In particular, $\D(R)$ satisfies $\RTC$ if and only if every pseudoflat epimorphism $R \to S$ is flat. 
\end{theorem}
\begin{proof}
    Recall from \cref{S:three} and \cref{S:four} that for the big tt-category $\D(R)$, both the Balmer spectrum $\Spec{\D(R)^\cpt}$ and the homological spectrum $\Spech{\D(R)}$ are homeomorphic to the Zariski spectrum $\Spec{R}$.

    $(i) \implies (ii):$ Since $f: R \to S$ is flat, it is easily identified with the right $\otimes$-idempotent associated to the definable $\otimes$-ideal $\D_{\Xcal_f} \cong \D(S)$ via \cref{ttftriples}. Let $V$ be the Thomason set corresponding to the hereditary torsion class $\Tcal = \Ker f$ of finite type via \cite[Theorem 2.2]{GP08}. In other words, for any $\pp \in \Spec R$ we have $\pp \not\in V$ if and only if $f \otimes_R R_\pp$ is an isomorphism if and only if $k(\pp) \otimes_R S \neq 0$. Then $\supph(S) = V^c$ and for each $\pp \in \supph(S)$ we have $R_\pp \in \D_{\Xcal_f}$. By \cref{Ep}(2), we have $\D_{\Xcal_f} \subseteq \D_V$. By \cref{cg-stalklocal}$(iv)$ and \cref{tc-stalklocal}, $\D_V = \Def{\bigoplus_{\pp \not\in V}R_\pp} \subseteq \D_{\Xcal_f}$, and thus $\D_{\Xcal_f} = \D_V$ is compactly generated.

    \noindent
    $(ii) \implies (i):$ If $\D_{\Xcal_f}$ is compactly generated then $\Xcal_f \cong \Mod S$ is closed under injective envelopes computed in $\Mod R$ by \cite[Lemma 3.3]{Hrb20}. It follows that $S$ is flat over $R$ by \cref{flat-injenv}.
\end{proof}
\begin{ex}
  \cref{flat-tc} gives a generalization of the equivalence $(1) \iff (2)$ of Bazzoni and Šťovíček \cite[Theorem 7.2]{BS17} to all commutative rings. It also shows that the result  \cite[Proposition 4.5]{AMSVT} of Angeleri-H\"{u}gel, Marks, Šťovíček, Vitória, and Takahashi follows alternatively from the fact that $\TC$ holds in $\D(R)$ for a commutative noetherian ring $R$, the result of Neeman \cite{Nee92}. 
\end{ex}
The following very slightly extends the setting of the counterexample to $\TC$ first constructed by Keller \cite{Kel94}. Here, note that a surjective ring epimorphism $R \to R/I$ is pseudoflat if and only if the ideal $I$ is \newterm{idempotent}, that is, $I=I^2$. Indeed, applying $R/I \otimes_R -$ to the exact sequence $0 \to I \to R \to R/I \to 0$ yields an isomorphism $I/I^2 \cong R/I \otimes_R I \cong \Tor^1_R(R/I,R/I)$.
\begin{prop}\label{Keller-ex}
  Let $R$ be a local ring with a proper non-zero idempotent ideal $I$. Then $\D(R)$ fails $\RTC$.
\end{prop}
\begin{proof}
  By the discussion above, $f: R \to R/I$ is a pseudoflat ring epimorphism. On the other hand, $f$ is not flat, because $R/I$ is not a flat $R$-module, see e.g. \cite[04PU]{Stacks}. Therefore, $\RTC$ fails by \cref{flat-tc}.
\end{proof}
\section{Separation in local rings}
In view of \cref{tc-scheme}, the study of $\TC$ in $\D(X)$ reduces to inspection of the definable $\otimes$-ideal $\Def{k}$ generated by the residue field $k=R/\mm$ of a local ring $(R,\mm)$. Let $\widehat{R} = \varprojlim_{n>0}R/\mm^n$ denote the $\mm$-adic completion of $R$. It is easy to observe that $\Def{k} = \Def{\widehat{R}}$: One inclusion follows from the isomorphism $\widehat{R} \otimes_R^\mathbf{L} k \cong k$, while the other follows from the Milnor triangle $\widehat{R} \to \prod_{n>0}R/\mm^n \to \prod_{n>0}R/\mm^n \mathrel{\leadsto}$ and the fact that $R/\mm^n$ is a finite extension of (possibly infinite) coproducts of copies of $R/\mm$. Thus, it is natural to study relations between $\TC$ and the $\mm$-adic topology. 

We say that $R$ is \newterm{($\mm$-adically) separated} if $\bigcap_{n>0}\mm^n = 0$, or equivalently, if the canonical map $R \to \widehat{R}$ is a monomorphism. Less standardly, let us say that $R$ is \newterm{purely separated} if the canonical map $R \to \widehat{R}$ is a pure monomorphism in $\Mod R$. For the theory of purity in module categories, we refer to \cite{Prest}, but we remark the standard fact that a short exact sequence in $\Mod R$ is pure if and only if the induced triangle is pure in $\D(R)$. In particular, any $R$ is purely separated whenever it is \newterm{complete}, meaning that the canonical map $R \to \widehat{R}$ is an isomorphism.

We extend the $n$-th ideal power $\mm^n$ to an arbitrary ordinal power following a similar idea as in \cite{Sto10}. Every ordinal $\alpha \geq \omega$ can be written uniquely as a sum $\lambda + n$ of a limit ordinal $\lambda$ and $n < \omega$. We define $\mm^\alpha$ inductively by putting $\mm^\lambda = \bigcap_{\beta < \lambda}\mm^\beta$ and $\mm^\alpha = (\mm^\lambda)^n$. We say that $R$ is \newterm{transfinitely separated} if there is an ordinal $\alpha$ such that $\mm^\alpha = 0$. We say that $R$ is \newterm{purely transfinitely separated} if it is transfinitely separated, and if for every limit ordinal $\alpha$ the canonical map $R/\bigcap_{\beta<\alpha}\mm^\beta \to \varinjlim_{\beta<\alpha}R/\mm^\beta$ is a pure monomorphism. Finally, we say that $R$ is \newterm{purely derived transfinitely separated} if it is transfinitely separated, and if for every limit ordinal $\alpha$ the canonical map $R/\bigcap_{\beta<\alpha}\mm^\beta \to \holim_{\beta<\alpha}R/\mm^\beta$ is a pure monomorphism in $\D(R)$, here $\holim$ is the right derived functor of $\varprojlim$. Note that by taking the zero cohomology, purely derived transfinitely separated implies purely transfinitely separated. The converse implication holds if $R$ is separated, because the inverse system $(R/\mm^n)$ consists of epimorphisms, and thus $\holim_{n>0}R/\mm^n \cong \varprojlim_{n>0}R/\mm^n$ since the Mittag-Leffler condition is satisfied.
\begin{lemma}\label{idemp-trans}
  A local ring $R$ is transfinitely separated if and only if there is no non-zero proper idempotent ideal in $R$.
\end{lemma}
\begin{proof}
  If there is a non-zero idempotent ideal $I$ of $R$ then, by definition, $I \subseteq \mm^\alpha$ for any ordinal $\alpha$, and so $R$ is not transfinitely separated. On the other hand, if no non-zero ideal is idempotent then we have for any ordinal $\beta$ that $\mm^\beta$ is non-zero that $\mm^\beta \supsetneq \mm^{\beta+1}$. Then there must be some ordinal $\alpha$ with $\mm^\alpha = 0$.
\end{proof}
\begin{lemma}\label{sep-tc}
  Let $R$ be a commutative ring. If $\D(R)$ satisfies $\RTC$ then $R_\pp$ is transfinitely separated for all $\pp \in \Spec{R}$.
\end{lemma}
\begin{proof}
  By \cref{tc-scheme} we can assume towards contradiction that $R$ is local and not transfinitely separated. Then combine \cref{idemp-trans} with \cref{Keller-ex}.
\end{proof}
\begin{lemma}\label{psep-tc}
  Let $R$ be a commutative ring. If $R_\pp$ is purely transfinitely separated for all $\pp \in \Spec R$ then $\D(R)$ satisfies $\RTC$. If $R_\pp$ is purely derived transfinitely separated for all $\pp \in \Spec R$ then $\D(R)$ satisfies $\TC$.

  In particular, if $R_\pp$ is purely separated (e.g., if $R_\pp$ is complete) for all $\pp \in \Spec R$ then $\D(R)$ satisfies $\TC$.
\end{lemma}
\begin{proof}
  By \cref{tc-scheme}, it is enough to show that for a purely derived transfinitely separated local ring $R$ with a residue field $k$ that $\Dcal := \Def{k} = \D(R)$. In the $\RTC$ case, we take $\Dcal$ to be the smallest definable wide subcategory (cf. \cref{rem-wide}) of $\Mod R$ containing $k$ instead. By induction on ordinal $\alpha$ we show that $R/\mm^\beta \in \Dcal$. If $\alpha = \beta+1$, $R/\mm^\alpha$ is an extension of $R/\mm^\beta$ and the $k$-module $\mm^\beta/\mm^\alpha$. Since $k \in \Dcal$, we have $R/\mm^\alpha \in \Dcal$. If $\alpha$ is a limit ordinal, we have by the assumption that $R/\mm^\alpha$ purely embeds into $\varprojlim_{\beta<\alpha}R/\mm^\beta$ in $\Mod R$ or to $\holim_{\beta<\alpha}R/\mm^\beta$ in $\D(R)$, depending on the assumption. Recall that $\Dcal$ is closed under pure monomorphisms and $\varprojlim$ (resp., $\holim$) when $\Dcal$ is a definable wide (resp. thick) subcategory of $\Mod R$ (resp., $\D(R)$). In either case, we get that $R/\mm^\alpha \in \Dcal$. Since there is $\alpha$ such that $R/\mm^\alpha = R$, we are done.
\end{proof}
We will see that neither \cref{sep-tc} nor \cref{psep-tc} can be reversed. In fact, we construct a separated local ring $(R,\mm,k)$ such that $\Def{k} \neq \D(R)$ in \cref{ex-1} and a separated, but not purely separated, zero-dimensional local ring satisfying $\TC$ in \cref{ex-2}.
\begin{lemma}\label{pure-fp}
  A local ring $R$ is purely separated if and only if each finitely presented $R$-module $F$ is separated, that is, the map $F \to \varprojlim_{n>0}F/\mm^n F$ is monic, or equivalently, $\bigcap_{n>0}\mm^n F = 0$.
\end{lemma}
\begin{proof}
  By definition of a pure monomorphism in $\Mod R$, the ring $R$ is purely separated if and only if $(R \to \widehat{R}) \otimes_R F$ is a monomorphism for any finitely presented $R$-module $F$. We prove the claim by showing that $\widehat{R} \otimes_R F$ is naturally isomorphic to $\varprojlim_{n>0}F/\mm^n F$. Let $f: R^m \to R^k$ be a free presentation of $F$. This induces an inverse system of morphisms
  $$
\begin{tikzcd}[sep = huge]
(R/\mm^n)^m \arrow{r}{R/\mm^n \otimes_R f} & (R/\mm^n)^k \\
(R/\mm^{n+1})^m \arrow{r}{R/\mm^{n+1} \otimes_R f} \arrow{u} & (R/\mm^{n+1})^k \arrow{u}
\end{tikzcd}
  $$
  where the vertical arrows are projections. Passing to the inverse limit yields the morphism $\widehat{R} \otimes_R f$ whose cokernel is $\widehat{R} \otimes_R F$.
  
  \noindent
  Let $I_n$ be the image of the map $R/\mm^n \otimes_R f$, and note that the commutative square above yields the induced epimorphism $I_{n+1} \to I_n$. We thus obtain an inverse system of short exact sequences
  $$0 \to I_n \to (R/\mm^n)^k \to F/\mm^n \to 0.$$
  Since the leftmost inverse system consists of epimorphisms, it is Mittag-Leffler, and so passing to the inverse limits yields a short exact sequence
  $$0 \to \varprojlim_{n>0} I_n \to \varprojlim_{n>0} (R/\mm^n)^k \to \varprojlim_{n>0} F/\mm^n \to 0.$$ Thus, we identified $\varprojlim_{n>0} F/\mm^n$ with $\widehat{R} \otimes_R F$.
\end{proof}
\begin{ex}
  If $R$ is a local noetherian ring then it is a standard fact that $R$ is purely separated, see e.g. \cite[00MC, 08WP]{Stacks}. In view of \cref{tc-scheme}, it follows that $\D(X)$ satisfies $\TC$ for any quasi-compact quasi-separated scheme with noetherian stalks. This recovers the $\TC$ results of Neeman \cite{Nee92} and Alonso, Jeremías, and Souto \cite[Corollary 5.4]{AJS04}, and of Stevenson \cite[Theorem 4.25]{Ste14} and Bazzoni and Šťovíček \cite[end of Section \S 7]{BS17} (a commutative ring is von Neumann regular precisely when all of its stalks are field).
\end{ex}
\begin{ex}
  Let $R$ be a valuation domain, that is, an integral domain whose ideals form a chain. We claim that $R$ is transfinitely separated if and only if $R$ is purely transfinitely separated. Indeed, assume in view of \cref{idemp-trans} that $R$ has no non-zero idempotent ideal. Let $\alpha$ be a limit ordinal and put $R' = R/\mm^\alpha$. For any $s \in R \setminus \mm^\alpha$ there is $\beta<\alpha$ such that $s \not\in \mm^\beta$. But then $s^n \not\in (\mm^\beta)^n$ by the fact that elements are totally ordered by divisibility in a valuation domain, and then  $s^n \not \in \mm^{\beta+n} \supseteq \mm^\alpha$. It follows that $R'$ is a domain, and so a valuation domain again. We can thus reduce to the case $\mm^\alpha = 0$.
  
  Now consider the monomorphism $R \to \varprojlim_{\beta < \alpha}R/\mm^\beta$. Since the ideals form a linear chain in a valuation domain, this map is the same as the $R$-adic completion $R \to \widetilde{R} = \varprojlim_{r \in R'}R'/rR'$ with respect to the topology generated by all non-zero ideals of $R'$. The completion $R \to \widetilde{R}$ is a pure monomorphism by \cite[\S VIII, Lemma 3.1]{FSbook}. It follows from \cref{psep-tc} that $\RTC$ holds for a valuation domain $R$ if and only if $R$ admits no non-zero idempotent ideal (which is equivalent to each localization $R_\pp$ having no non-zero idempotent ideal for a valuation domain), such valuation domains are called \newterm{strongly discrete}. Since $\RTC$ is equivalent to $\TC$ for valuation domains by \cite[Theorem 3.10]{BS17}, this recovers \cite[Theorem 7.2]{BS17}.
\end{ex}
\begin{lemma}\label{idemp-factor}
  Let $R$ be a local ring and $I$ a finitely generated ideal of $R$ such that $R/I$ is not transfinitely separated. Then $\Deft{k} \neq \D(R)$, and thus $\D(R)$ fails $\TC$.

  Moreover, if $I$ is a principal ideal generated by a non-zerodivisor of $R$ then $\D(R)$ fails $\RTC$.
\end{lemma}
\begin{proof}
  By \cref{idemp-trans}, there is a non-trivial idempotent ideal $\overline{J}$ of $R/I$, let $J$ denote its full preimage in $R$. Define
  $$\Dcal = \left\{X \in \D(R) \middle\vert \begin{array}{c}\Hom_{\D(R)}(K(I),\Sigma^n X) \xrightarrow{\cdot j} \Hom_{\D(R)}(K(I),\Sigma^n X) \\ \text{ is a zero map $\forall j \in J, n \in \Zbb$} \end{array} \right\}.$$
  We claim that $\Dcal$ is a definable $\otimes$-ideal. Since $K(I)$ is a compact object, it is easy to check that $\Dcal$ is closed under pure subobjects. Clearly, $\Dcal$ is closed under all products and shifts. It remains to show that $\Dcal$ is closed under extensions (recall that any localizing subcategory of $\D(R)$ is closed under tensoring). Let $X \to Y \to Z \to \Sigma X$ be a triangle such that $X, Z \in \Dcal$. We need to show that the multiplication by any $j \in J$ on $\Hom_{\D(R)}(K(I),Y)$ is zero. Since $\overline{J}$ is idempotent in $R/I$, there are $j_1^k, j_2^k \in J$ for $k=1,\ldots,l$ and $i \in I$ such that $j = \sum_{k=1}^l j_1^k j_2^k + i$. Consider the following commutative diagram with exact rows:
  $$
  \begin{tikzcd}
    \Hom_{\D(R)}(K(I),X) \arrow{r} \arrow{d}{\cdot j_1^k} & \Hom_{\D(R)}(K(I),Y) \arrow{d}{\cdot j_1^k} \arrow{r} & \Hom_{\D(R)}(K(I),Z) \arrow{d}{\cdot j_1^k} \\
    \Hom_{\D(R)}(K(I),X) \arrow{d}{\cdot j_2^k} \arrow{r} & \Hom_{\D(R)}(K(I),Y) \arrow{d}{\cdot j_2^k} \arrow{r} & \Hom_{\D(R)}(K(I),Z) \arrow{d}{\cdot j_2^k} \\
    \Hom_{\D(R)}(K(I),X) \arrow{r} & \Hom_{\D(R)}(K(I),Y) \arrow{r} & \Hom_{\D(R)}(K(I),Z) \\
  \end{tikzcd}
  $$
  Since $X,Z \in \Dcal$, both the vertical maps on the left and on the right are zero maps. A simple diagram chasing shows that the map $\Hom_{\D(R)}(K(I),Y) \xrightarrow{\cdot j_1^k} \Hom_{\D(R)}(K(I),Y)$ factorizes through $\Hom_{\D(R)}(K(I),X)$, and then the composition $\Hom_{\D(R)}(K(I),Y) \xrightarrow{\cdot j_1^k} \Hom_{\D(R)}(K(I),Y) \xrightarrow{\cdot j_2^k} \Hom_{\D(R)}(K(I),Y)$ is zero. Since $\Hom_{\D(R)}(K(I),Y)$ is annihilated by $I$, the sum of compositions $\sum_{k=1}^l j_1^k j_2^k$ is equivalent to $\Hom_{\D(R)}(K(I),Y) \xrightarrow{\cdot j} \Hom_{\D(R)}(K(I),Y)$.

  \noindent
  Now clearly the residue field $k$ at the maximal ideal of $R$ belongs to $\Dcal$. On the other hand, $R \not\in \Dcal$. Indeed, this would mean that the multiplication map $K(I)^* \xrightarrow{\cdot j} K(I)^*$ induces a zero map on all cohomologies, where $K(I)^* = \RHom_R(K(I),R)$ is the dual Koszul complex. This would imply that the map $R/I  \xrightarrow{\cdot j} R/I$ is zero for all $j \in J$, which cannot be true for any $j \in J \setminus I$. Then we have $\Def{k} \subseteq \Dcal \subsetneq \D(R)$, and so $\D(R)$ fails $\TC$ by \cref{tc-scheme}.

  To prove the final claim, let $I$ be a principal ideal generated by a non-zerodivisor $x \in R$. Then $\Dcal$ is closed under cohomology by the Künneth formula, and so $\D(R)$ fails $\RTC$. In fact, the bireflective extension-closed subcategory of $\Mod R$ corresponding to $\Dcal$ via \cref{epi-biloc} is given by all $R$-modules $M$ such that both the kernel and the cokernel of the multiplication map $M \xrightarrow{\cdot x}M$ are annihilated by $J$.
\end{proof}

\begin{ex}\label{ex-1}
  There is a separated local integral domain $(R,\mm,k)$ with $x \in R$ such that the maximal ideal $\overline{\mm}$ of $R/(x)$ is non-zero and idempotent. Indeed, let $\Zbb_2[[x]]$ be the ring of formal power series of the two element field $\Zbb_2$ and consider the subset 
  $$B=\{\sum_{i=1}^\infty x^{2^{n_i}} \mid (n_i)_{i > 0} \text{ strictly increasing sequence of non-negative integers}\}$$
  of $\Zbb_2[[x]]$. Note that any element $\sum_{i=1}^\infty x^{2^{n_i}} \in B$ such that $n_1>0$ has a square root $\sum_{i=1}^\infty x^{2^{n_i-1}} \in B$ in $\Zbb_2[[x]]$. Since $B$ is of uncountable cardinality, there is $\alpha_0 \in B$ which is transcendental over $\Zbb_2[x]$, that is, if $F(y) \in \Zbb_2[x][y]$ is such that $F(\alpha_0) = 0$ then $F=0$. Let us define a sequence of elements $\alpha_n \in B$ for $n>0$ inductively as follows: If $\alpha_n$ has a non-vanishing linear monomial, let $f_n = x$, otherwise let $f_n = 0$. Then there is $\alpha_{n+1} \in B$ such that $\alpha_{n+1}^2 = \alpha_n + f_n$. Let $$R = \Zbb_2[x,\alpha_n : n \geq 0]_{(x,\alpha_n : n \geq 0)}$$ be the subalgebra of $\Zbb_2[[x]]$ generated by the elements $x$ and $\alpha_n$ for $n \geq 0$ and localized at the maximal ideal $(x,\alpha_n : n \geq 0)$, the intersection of the maximal ideal $(x)$ of $\Zbb_2[[x]]$ with $\Zbb_2[x,\alpha_n : n \geq 0]$. 
  
  Then the inclusion $R \subseteq \Zbb_2[[x]]$ is a local ring morphism, in particular, $R$ is separated as $\Zbb_2[[x]]$ is local noetherian. We claim that the maximal ideal $\overline{\mm}$ of $R/(x)$ is non-zero and idempotent. The idempotency follows from the construction as $\overline{\mm}$ is generated by the images of elements $\alpha_n, n \geq 0$ and we get the relations $\alpha_{n+1}^2 = \alpha_n$ in $R/(x)$. It remains to show that the maximal ideal $\mm$ of $R$ is not generated by $x$. To this end, we will show that $x^{-1}\alpha_0 \not\in R$. The relations yield that for each $m<n$, $\alpha_m = \alpha_n^{2^{n-m}} + x H_{m,n}(x,\alpha_n)$ for some polynomial $H(x,y) \in \Zbb_2[x,y]$. If $x^{-1}\alpha_0 \in R$ then there are polynomials $F(x,y),G(x,y) \in \Zbb_2[x,y]$ with $F(x,y)$ having a non-zero constant term such that $x^{-1}(\alpha_n^{2^n} + x H_{0,n}(x,\alpha_n)) = \frac{G(x,\alpha_n)}{F(x,\alpha_n)}$ for some $n \geq 0$. Then we get the following polynomial relation in $\alpha_n$ over $\Zbb_2[x]$
  $$\alpha_n^{2^n}F(x,\alpha_n) + x (H_{0,n}(x,\alpha_n)F(x,\alpha_n) + G(x,\alpha_n)) = 0$$
  that is non-trivial due to $F(x,y)$ having a non-zero constant term, this contradicts $\alpha_n$ being transcendental over $\Zbb_2[x]$.

  By \cref{idemp-factor}, $\D(R)$ fails $\RTC$. Then there is a non-flat pseudoflat ring epimorphism $f: R \to S$ by \cref{flat-tc}. This is a local ring morphism (as $\mm \in \supph(S)$) which is an epimorphism but is not surjective, the last property follows from $R$ having no non-zero idempotent ideal by \cref{idemp-trans} and the discussion before \cref{Keller-ex}. Lazard \cite{Laz69} (see also \cite[06RH]{Stacks}) constructed a non-surjective local ring epimorphism over a ring of Krull dimension zero. We do not know whether such an example which is additionally pseudoflat can be realized over a zero-dimensional ring --- Lazard's example cannot be pseudoflat by \cref{ex:trunc}.
\end{ex}
\subsection{Zero-dimensional local rings}
Finally, we limit the focus of our study even more and consider the case of a zero-dimensional local ring $R$, that is, an affine scheme with a single point. Here, \cref{tc-scheme} yields a particularly simple characterization of $\TC$: It holds in $\D(R)$ if and only if $\Def{k} = \D(R)$, where $k$ is the unique residue field. 

Recall that $R$ is called \newterm{coherent} if the finitely presented $R$-modules form a wide subcategory of $\Mod R$. Following Colby \cite{Co75}, an $R$-module $M$ is called \newterm{$\aleph_0$-injective} if $\Ext_R^1(R/I,M) = 0$ for all finitely generated ideals $I$ of $R$. We say that $R$ is \newterm{self-$\aleph_0$-injective} is it is $\aleph_0$-injective over itself.
\begin{lemma}
  Let $R$ be a local ring which is coherent, and such that $R$ is self-$\aleph_0$-injective. Then $R$ is separated if and only if it is purely separated.
\end{lemma}
\begin{proof}
  By \cite[Theorem 1, Theorem 2]{Co75}, the assumption on $R$ ensures that every finitely presented $R$-module is a submodule in a free $R$-module of finite rank. Then it is easy to see that $R$ separated implies that each finitely presented $R$-module is separated, and so $R$ is purely separated by \cref{pure-fp}.
\end{proof}
\begin{ex}\label{ex:colby}
  Let $R$ be a local ring that can be written as a direct limit $R = \varinjlim_{i \in I}R_i$ of commutative rings such that each $R_i$ is coherent and self-$\aleph_0$-injective and the transition maps $R_i \to R_j$ are all flat for any $i<j$ in $I$. Then $R$ is coherent by \cite[Chap. I, \S 2, p. 62]{BouAC}. We claim that $R$ is also self-$\aleph_0$-injective. Indeed, let $I$ be a finitely generated ideal of $R$ and a map $f: I \to R$. Then there is $i \in I$, a finitely generated ideal $J$ of $R_i$, and a map $g: J \to R_i$, such that $f = g \otimes_{R_i} R$. Since $R_i$ is self-$\aleph_0$-injective, $g$ extends to a map $h: R_i \to R_i$. Then $h \otimes_{R_i} R$ extends $f$ to a map $R \to R$.
\end{ex}
\begin{ex}\label{ex:trunc}
  Let $K$ be a field and $R = K[x_0,x_1,x_2,\ldots]/(x_i^{n_i} : i \geq 0)$ be the truncated polynomial ring for some choice of positive integers $n_i$. We claim that $R$ is purely separated. This implies $\TC$ for $\D(R)$ by \cref{psep-tc}, recovering \cite[Corollary 7.5]{DP08}.

  To demonstrate that $R$ is purely separated, we use \cref{ex:colby}. Indeed, $R = \varinjlim_{n \geq 0}R_k$, where $R_k = K[x_0,x_1,x_2,\ldots,x_k]/(x_i^{n_i} : i=0,1\ldots,k)$. For each $k \geq 0$, the ring $R_k$ is a local Gorenstein artinian ring, that is, a local noetherian self-injective ring. Clearly, the transition maps $R_k \to R_m$ are flat (even projective) for all $k<m$.
\end{ex}
Finally, we provide an example of a separated, but not purely separated, zero-dimensional local ring which satisfies $\TC$.
\begin{lemma}\label{0-dim-tc-quot}
  Let $R$ be a local zero-dimensional ring and assume there is a finitely generated ideal $I$ of $R$ such that $\D(R/I)$ satisfies $\TC$. Then $\D(R)$ satisfies $\TC$.
\end{lemma}
\begin{proof}
  Let $\Dcal = \Def{k}$, where $k$ is the unique residue field of $R$. It suffices by \cref{tc-scheme} to show that $R/I \in \Dcal$. Indeed, since $I$ is finitely generated and $R$ is zero-dimensional, $I$ is nilpotent. It follows that $R$ is obtained as a finite extension of $R/I$-modules, and so $R \in \Dcal$, yielding $\Dcal = \D(R)$.

  \noindent
  Since $\D(R/I)$ satisfies $\TC$ and $k$ is also the residue field of the local ring $R/I$, we have $\overline{\Dcal} = \D(R/I)$ by \cref{tc-scheme}, where $\overline{\Dcal}$ is the smallest definable $\otimes$-ideal in $\D(R/I)$ containing $k$. Then also $R/I \in \Dcal$. Indeed, it is straightforward to check that $(R \to \lambda_\Dcal(R)) \otimes_R^\mathbf{L} R/I$ is a right $\otimes$-idempotent in $\D(R/I)$, and thus identifies with $R/I \to \lambda_{\overline{\Dcal}}$. Since $R/I \to \lambda_{\overline{\Dcal}} \cong \lambda_\Dcal(R) \otimes_R^\mathbf{L} R/I$ is an isomorphism in $\D(R/I)$, it is also an isomorphism in $\D(R)$.
\end{proof}
\begin{ex}\label{ex-2}
  Let $K$ be a field and consider the truncated polynomial ring $S = K[x_1,x_2,\ldots]/(x_i^2: \forall i>0)$. Put 
  $$T = S[z,y]/(z^2,y^2,z-x_{i+1}y+x_1\cdots x_i: \forall i>0).$$ 
  Let $\nn$ be the maximal ideal of $T$ and put $R = T/\nn^\omega$, then $R$ is a zero-dimensional and separated local ring. Note that $T/(y,z) \cong K[x_2,x_3,\ldots]/(x_i^2: \forall i>1)$ is again a truncated polynomial ring and thus purely separated by \cref{ex:trunc}. In particular, $\nn^\omega \subseteq (y,z)$, and so $R/(y,z) \cong K[x_2,x_3,\ldots]/(x_i^2: \forall i>1)$ is also purely separated. On the other hand, $R/(y)$ is not separated. Indeed, we have the relations $z = x_1\cdots x_i$ for all $i>0$ in $R$, and so it suffices to notice that $z$ is not in $\nn^\omega$ as element of $T$. For this, let us inspect the quotient $T/(x_1) \cong K[y,x_2,x_3,\ldots]/(y^2,x_i^2,x_iy-x_jy: \forall j>i>1)$. This is clearly a separated ring as we are modding out homogenous relations in which the image of $z$ is equal to the non-zero element $x_2 y$.
  
  It follows that $R$ is not purely separated by \cref{pure-fp}, and thus $R$ is a separated, but not purely separated, zero-dimensional local ring which satisfies $\TC$ by \cref{0-dim-tc-quot}.
\end{ex}

\bibliographystyle{amsalpha}
\bibliography{bibitems}
\end{document}